\documentclass[11pt,article]{article}

\usepackage{amsmath}
\usepackage{amsthm}
  \usepackage{paralist}
  \usepackage{graphics} 
  \usepackage{epsfig} 
\usepackage{graphicx}
\usepackage{caption}
\usepackage{subcaption}
\usepackage{epstopdf}
 \usepackage[colorlinks=true]{hyperref}
 \usepackage{multirow}
\usepackage{setspace}

\input{amssym.tex}

\newtheorem{theorem}{Theorem}[section]
\newtheorem{corollary}{Corollary}
\newtheorem{lemma}[theorem]{Lemma}
\newtheorem{proposition}{Proposition}

 \numberwithin{equation}{section}
\newtheorem{remark}{Remark}

\newcommand{\keywords}

\def\bc{\begin{center}}       \def\ec{\end{center}}
\def\ba{\begin{array}}        \def\ea{\end{array}}
\def\be{\begin{equation}}     \def\ee{\end{equation}}
\def\bea{\begin{eqnarray}}    \def\eea{\end{eqnarray}}
\def\beaa{\begin{eqnarray*}}  \def\eeaa{\end{eqnarray*}}

\def\mathbb{\Bbb}

\begin{document}
\newpage

\textbf{Title}: Qualitative studies of advective competition system with Beddington--DeAngelis functional response\\

\textbf{Running title}: Beddington--DeAngelis competition system\\

\textbf{Authors}: Ling Jin, Qi Wang, Zengyan Zhang\\

\textbf{Corresponding author}: Qi Wang\\

\textbf{Email}: qwang@swufe.edu.cn\\

\textbf{Address}: Department of Mathematics\\
Southwestern University of Finance and Economics\\
555 Liutai Ave, Wenjiang\\
Chengdu, China 611130

\newpage

\title{\bf Qualitative studies of advective competition system with Beddington--DeAngelis functional response}
\author{Ling Jin, Qi Wang \thanks{Corresponding author.  Email: qwang@swufe.edu.cn.  QW receives supports from NSFC--China (No.11501460), the Project sponsored by SRF
for ROCS, SEM and the Project (No.15ZA0382) from Department of Education, Sichuan China.}, Zengyan Zhang \\
Department of Mathematics\\
Southwestern University of Finance and Economics\\ 
Chengdu, China 611130
} 

\date{}
\maketitle

\abstract
This paper investigates a reaction--advection--diffusion system modeling interspecific competition between two species over bounded domains.  The kinetic terms are assumed to satisfy the Beddington--DeAngelis functional responses.  We consider the situation that first species disperse by a combination of random walk and directed movement along the population density of the second species which disperse randomly within the habitat.  For multi--dimensional bounded domains, we prove the global existence and boundedness of time--dependent solutions.  For one--dimensional finite domains, we study the effect of diffusion and advection on the existence and stability of nonconstant positive steady states to the strongly coupled elliptic system.  In particular, our stability result of these nontrivial steady states provides a selection mechanism for stable wavemodes of the time-dependent system.  In the limit of diffusion rates, we show that the steady states of this full elliptic system can be approximated by nonconstant positive solutions of a shadow system and then we construct boundary spike solutions to this shadow system.  For the full elliptic system, we also investigate solutions with a single boundary spike or an inverted boundary spike, i.e., the first species concentrate on the boundary point while the second species dominate the whole habitat except the boundary point.  These spatial structures can be used to model the spatial segregation phenomenon through interspecific competitions.  Some numerical studies are performed to illustrate and support our theoretical findings.

2010 \emph{Mathematics Subject Classification}. {Primary: 35B25, 35B35, 35B40, 35A01, 35J47.  Secondary: 35B32, 92D25,92D40.}\\
\textbf{Keywords: Competition system, Beddington--DeAngelis functional response, global solutions, steady--state, stability, boundary spike, boundary layer.}

\section{Introduction}
In this paper, we consider the following reaction--advection--diffusion system
\begin{equation}\label{1}
\left\{
\begin{array}{ll}
u_t=\nabla \cdot (D_1\nabla u+\chi u \phi(v)\nabla v)+f(u,v),&x \in \Omega,t>0,\\
v_t=D_2\Delta v+g(u,v),&x \in \Omega,t>0,\\
\partial_\textbf{n}u=\partial_\textbf{n}v=0,&x\in\partial \Omega,t>0,\\
u(x,0)=u_0(x), v(x,0)=v_0(x),&x\in \Omega,
\end{array}
\right.
\end{equation}
where
\begin{equation}\label{2}
f(u,v)=\Big(-\alpha_1+\frac{\beta_1}{a_1+b_1u+c_1v}\Big)u,~g(u,v)=\Big(-\alpha_2+\frac{\beta_2}{a_2+b_2u+c_2v}\Big)v.
\end{equation}
$\Omega$ is a bounded smooth domain in $\mathbb R^N$, $N \geq 1$, $\partial\Omega$ is its boundary and $\textbf{n}$ denotes the unit outer normal on $\partial\Omega$.  $D_i$, $\alpha_i$, $\beta_i$, $a_i$, $b_i$ and $c_i$, $i=1,2$, are positive constants.  $\chi$ is a constant and $\phi$ is a smooth function such that $\phi(v)>0$ for all $v>0$.  The initial conditions $u_0(x)$ and $v_0(x)$ are non-negative smooth functions which are assumed to be not identically zero.

System (\ref{1}) describes the evolution of population densities of two mutually interfering species over a bounded habitat $\Omega$.  $u(x,t)$ and $v(x,t)$ are population densities of the competing species at space--time location $(x,t)\in\Omega\times \mathbb R^+$.  Diffusion rates $D_1$ and $D_2$ measure unbiased dispersals of the species.  It is assumed that species $u$ senses the population pressure from $v$ and directs its dispersal accordingly.  In particular, $u$ moves up or down along the population gradient of interspecies $v$ and $\chi$ is a constant that measures the intensity of such directed movement.  $\chi<0$ if $u$ invades the dwelling habitat of $v$ and $\chi>0$ if $u$ escapes the habitat of $v$.  Therefore $u$ takes active dispersal strategy to cope with population pressure from $v$, either to seek or avoid interspecific competition.  $\phi$ reflects the variation of the directed movement strength with respect to population density $v$.

For almost all mechanistic models that describe population dynamics, functional responses play key roles in the spatial--temporal and qualitative properties of the population distributions.  One feature of these kinetics is to take into account the environment's maximal load or carrying capacity.  Among the most common types are Lotka--Volterra \cite{Ta, BT}, Hollings \cite{KR, YWS} and Leslie--Gower \cite{AO}, etc.  To investigate mutual interference among intra--species, Beddington \cite{Bed} and DeAngelis \emph{et al}. introduced \cite{De} the following functional response of focal species $u$ preying natural resources
\[f_0(R,u)=\frac{1}{1+ahR+bu},\]
where $R$ is the resource density, $a$ is the attacking rate and $h$ is the handling time, while $b$ measures the interference rate.  To account for both intra-- and inter--specific competition, B. de Villemereuil and Lopez-Sepulcre \cite{VL} generalized the B-D response into
\[f_0(R,u,v)=\frac{1}{1+ahR+bu+cv},\]
where $b$ and $c$ represent the intra- and inter--specific competition rates.  Moreover, they performed field experiments on guppy-killifish system and collected data that the above models in the absence and presence of inter--species.  Functional response $f_0(R,u,v)$ models a competition relationship between $u$ and $v$ based on the idea that an increase in the population density one species should decreases the growth rate of all individuals since they consume the same resources.

In this paper, we assume that species $u$ and $v$ satisfy the Beddington--DeAngelis functional responses in (\ref{2}).  Ecologically, $\alpha_1$ and $\alpha_2$ represent intrinsic death rates of species $u$ and $v$ and $\beta_1$ and $\beta_2$ interpret their growth rates; $a_1$ and $a_2$ represent the compound effects of resource handling time and attack rate, while $b_1$ and $c_2$ account for intra-specific competition and $b_2$ and $c_1$ are the coefficients of inter--specific competition.  We are motivated to investigate the dynamics of population distributions to (\ref{1}) and its stationary system due to the effect of biased movement of species $u$.  To manifest this effect, we assume that the resources are spatially homogeneous and all the parameters are assumed to be positive constants.

To model the coexistence and segregation phenomenon through interspecific competition, various reaction--diffusion systems with advection or cross--diffusion have been proposed and studied.  For example, Wang \emph{et al.} investigated in \cite{WGY} the global existence of (\ref{1}) with Lotka--Volterra kinetics $f(u,v)=\big(a_1-b_1u-c_1v\big)u$ and $g(u,v)=\big(a_2-b_2u-c_2v \big)v$.  Existence and stability of nonconstant positive steady states have also been established through rigour bifurcation analysis.  They also studied transition layer steady states to the system.  Another example of this type is the SKT model proposed by Shigesada, Kawasaki and Teramoto \cite{SKT} in 1979 to study the directed dispersals due to mutual interactions.  See \cite{LN, LN2, MK, NWX, SKT} etc. for works and recent developments on the SKT model.  It is also worthwhile to mention that predator--prey models with Beddington--DeAngelis type functional response have been extensively used in \cite{CC,FK,Hw} etc.  Moreover, we refer the reader to \cite{CDAO} and the recent survey paper \cite{C} for detailed discussions on reaction--advection--diffusion systems of population dynamics.

From the viewpoint of mathematical modeling, it is interesting and important to investigate the spatially inhomogeneous distribution of population densities such as the coexistence and segregation of mutually interfering species, in particular, due to the effect of the dispersal strategy and population kinetics.  Time--dependent solutions can be used to model the segregation phenomenon in terms of finite or infinite blow--ups, that is, a species population density converges to a $\delta$--function or a linear combination of $\delta$--functions.  This approach has been taken for Keller--Segel chemotaxis system that models the directed cellular movements, along the gradient of chemicals in their environment.  See \cite{HP,HV,Nan}.  An alternative approach is to show that the solutions exist globally and converge to bounded steady states.  Then positive steady states with concentrating or aggregating structures such as boundary spikes, transition layers, etc. can be used to model the segregation phenomenon.  This approach has been taken by \cite{LN2, NWX, WX} etc.  The blow--up solution or a $\delta$--function is evidently connected to the species segregation phenomenon, however, it is not an optimal choice from the viewpoint of mathematical modeling since it challenges the rationality that population density can not be infinity.  On the other hand, it brings challenges to numerical simulations and makes it impossible to analyze the states after blow--up.

The rest of this paper is organized as follows.  In Section \ref{section2}, we obtain the existence and uniform boundedness of positive classical solutions to (\ref{1}) over multi--dimensional bounded domains--See Theorem \ref{theorem21}.  Our proof begins with the local existence and extension theory of Amann \cite{A,A1} for general quasilinear parabolic systems.  In Section \ref{section3}, we study the existence and stability of nonconstant positive steady states of (\ref{1}) over $\Omega=(0,L)$.  Our method is based on the local theory of Crandall and Rabinowitz \cite{CR} and its new version recently developed by Shi and Wang in \cite{SW}.  Our stability results give a selection mechanism for stable wavemodes of system (\ref{1})--See Theorem \ref{theorem32} and Theorem \ref{theorem33}.  Section \ref{section4} is devoted to the existence and asymptotic behaviors of nonconstant steady states with large amplitude.  It is shown that (\ref{1})--(\ref{2}) admits boundary spike and boundary layer solution if $D_1$, $\chi$ are comparably large and $D_2$ is small.  In Section \ref{section5} we briefly discuss our results and their applications.  Some interesting problems are also proposed for future studies.

In the sequel, $C$ and $C_i$ denote positive constants that may vary from line to line.

\section{Global existence and boundedness}\label{section2}
In this section, we investigate the global existence of positive classical solutions to system (\ref{1})--(\ref{2}).  We will show that the $L^\infty$-norms of $u$ and $v$ are both uniformly bounded in time.  The first set of our main results states as follow.
\begin{theorem}\label{theorem21}
Let $\Omega \subset \mathbb R^N$, $N\geq 1$, be a bounded domain with smooth boundary $\partial\Omega$ and assume that $\phi$ is a smooth function satisfying $\phi(v)>0$ for all $v\geq0$. Suppose that $(u_0,v_0)\in C(\bar \Omega)\times W^{1,p}$ for some $p>N$, and $u_0$, $v_0\geq0$, $\not \equiv0$ on $\bar \Omega$.  Then the IBVP (\ref{1})--(\ref{2}) admits a unique globally bounded classical solution $(u(x,t),v(x,t))$; moreover $u$ and $v$ are positive on $\bar \Omega$ for all $t\in(0,\infty)$.
\end{theorem}
Many reaction--diffusion systems of population dynamics have maximum principles, which can be used to prove the global existence and boundedness of their classical solutions.  However, the presence of advection term makes system (\ref{1}) non--monotone and it inhibits the application of the comparison principle, at least in proving the boundedness of $u$.  A review of literature suggests that there are two well--established methods to prove the global existence for reaction--advection--diffusion systems.  One method is to derive the $L^\infty$ through the $L^p$-iterations of some compounded functions of the population densities; another method is to apply standard theory on semigroups generated by $\{e^{-\Delta t}\}_{t\geq0}$ and $L^p$--$L^q$ estimates on the abstract form of the system.  Our proof of global existence to (\ref{1})--(\ref{2}) involves both techniques.

\subsection{Preliminary results and Local existence}

First of all, we collect some basic properties of the analytic Neumann semigroup $\{e^{t\Delta}\}_{t\geq0}$.  We refer the reader to \cite{Henry} for classical results and \cite{H, HW, Winkler2} for recent developments.  Let $\Omega \subset \mathbb R^N$, $N\geq1$, be a bounded domain.  It is well known that $A=-\Delta+1$ is sectorial in $L^p(\Omega)$ and it possesses closed fractional powers $A^\theta, \theta \in (0,1)$.  For $m=0,1$, $p\in[1,\infty]$ and $q\in(1,\infty)$, the domain $\mathcal{D}(A^\theta)$ is a Banach space endowed with norm
\[\Vert w \Vert_{\mathcal{D}(A^\theta)} =\Vert A^\theta w \Vert_{L^p(\Omega)};\]
moreover, we have the following embeddings
\begin{equation}\label{21}
\mathcal{D}(A^\theta) \hookrightarrow
\left\{
\begin{array}{ll}
W^{m,q}(\Omega),&\text{ if } m-N/q<2\theta-N/p, \\
C^\delta(\bar{\Omega}),& \text{ if } 2\theta-N/p >\delta\geq0.
\end{array}
\right.
\end{equation}
furthermore, for each $p>1$, $\{e^{t\Delta}\}_{t\geq 0}$ maps $L^p$ into $\mathcal{D}(A^\theta)$: there exists a positive constant $C$ dependent on $\Omega$, $N$ and $p$ and $\nu$ such that, for any $q\in(p,+\infty]$
\begin{equation}\label{22}
\Vert A^\theta e^{-At} w \Vert_{L^q(\Omega)} \leq C t^{-\theta-\frac{N}{2}(\frac{1}{p}-\frac{1}{q})} e^{-\nu t}\Vert w \Vert_{L^p(\Omega)}, \forall~w \in L^p(\Omega),
\end{equation}
and
\begin{equation}\label{23}
\Vert A^\theta e^{t\Delta} w \Vert_{L^q(\Omega)} \leq C t^{-\theta-\frac{N}{2}(\frac{1}{p}-\frac{1}{q})} e^{-\nu t}\Vert w \Vert_{L^p(\Omega)}, \forall~w \in L^p(\Omega),\int_\Omega w=0,
\end{equation}
where $\nu>0$ is the principal Neumann eigenvalue of $-\Delta$.

Next we present the local existence of classical solutions to (\ref{1})--(\ref{2}) and their extension criterion based on Amann's theory in the following theorem.
\begin{theorem}\label{theorem22}
Let $\Omega$ be a bounded domain in  $\mathbb{R}^N$, $N\geq1$ with smooth boundary $\partial \Omega$ and assume that $\phi$ is a continuous function.  Then for any initial data $(u_0,v_0)\in C(\bar\Omega)\times W^{1,p}(\Omega)$, $p>N$, satisfying $u_0, v_0\geq, \not\equiv 0$ on $\bar \Omega$, there exists a constant $T_{\max}\in(0,\infty]$ and a unique solution $(u(x,t),v(x,t))$ to (\ref{1})--(\ref{2}) defined on $\bar \Omega\times [0,T_{\max})$ such that $(u(\cdot,t)$, $v(\cdot,t)) \in C^0(\bar \Omega,[0,T_{\max}))$, $u,v \in C^{2,1}(\bar \Omega,[0,T_{\max}))$, and $u(x,t)$, $v(x,t)>0$ on $\bar \Omega$ for all $t\in(0,T_{\max})$.  Moreover, if $\sup_{s\in(0,t)}\Vert (u,v)(\cdot,s) \Vert_{L^\infty}$ is bounded for $t\in(0,T_{\max})$, then $T_{\max}=\infty$, \emph{i.e.}, $(u,v)$ is global in time.
\end{theorem}
\begin{proof}
Let $\textbf{w}=(u,v)$.  (\ref{1}) can be rewritten as
\begin{equation}\label{24}
\left\{
\begin{array}{ll}
\textbf{w}_t=\nabla \cdot(A(\textbf{w})\nabla \textbf{w})+F(\textbf{w}),~x \in \Omega,~t>0,     \\
\textbf{w}(x,0)=(u_0,v_0),x\in \Omega; \frac{\partial \textbf{w}}{\partial \textbf{n}}=0,~x \in \partial \Omega,t>0,
\end{array}
\right.
\end{equation}
where
 \begin{equation*}A(\textbf{w})=\begin{pmatrix}
 D_1  &  \chi u\phi(v)  \\
  0           &  D_2
  \end{pmatrix} ,~~F(\textbf{w})=\begin{pmatrix}
f(u,v) \\
g(u,v)
  \end{pmatrix}.
   \end{equation*}
(\ref{24}) is a triangular \emph{normally parabolic} system since the eigenvalues of $A$ are positive, then the existence part follows from Theorem 7.3 and Theorem 9.3 of \cite{A} and the extension criterion follows from Theorem 5.2 in \cite{A1}.  Moreover, one can apply the standard parabolic boundary $L^p$ estimates and Schauder estimates to see that $u_t$, $v_t$ and all the spatial partial derivatives of $u$ and $v$ are bounded in $\bar \Omega \times (0,\infty)$ up to the second order, hence $(u,v)$ has the regularities as stated in the Theorem.

On the other hand, we can use parabolic Strong Maximum Principle and Hopf's boundary point lemma to show that $u>0$ and $v>0$ on $\bar \Omega \times (0,T_{\max})$.  This completes the proof of Theorem \ref{theorem22}.
\end{proof}

\subsection{A prior estimates}
We collect some properties of the local classical solutions obtained in Theorem \ref{theorem22}.  First of all, we have the following results.
\begin{lemma}\label{lemma23}Under the same conditions as in Theorem \ref{theorem22}, there exists a positive constant $C_1$ dependent on $a_1, b_1, \alpha_1, \beta_1, \vert \Omega \vert$ and $\Vert u_0 \Vert_{L^1}$ such that
\begin{equation}\label{25}
\Vert u(\cdot,t) \Vert_{L^1(\Omega)} \leq C_1, \forall t\in(0,T_{\max});
\end{equation}
moreover, for any $p\in(1,\infty)$, there exists a positive constant $C(p)$ dependent on $a_2, b_2, \alpha_2, \beta_2, \vert \Omega \vert$ and $\Vert v_0 \Vert_{L^p}$ such that
\begin{equation}\label{26}
\Vert v(\cdot,t) \Vert_{L^p(\Omega)}\leq C(p),\forall t\in(0,T_{\max}).
\end{equation}
\end{lemma}
\begin{proof}
To show (\ref{25}), it is sufficient to show that $\int_\Omega u(x,t) dx$ is uniformly bounded for $t\in(0,\infty)$ since $u(x,t)>0$ according to Theorem \ref{theorem22}. Integrating the first equation in (\ref{1}) over $\Omega$ yields
\begin{align}\label{27}
 \frac{d}{dt} \int_\Omega u(x,t)dx &=  \int_\Omega \big(-\alpha_1+\frac{\beta_1}{a_1+b_1u+c_1v}\big)udx \nonumber \\
&= -\alpha_1\int_\Omega  udx+\int_\Omega \frac{\beta_1 u}{a_1+b_1u+c_1v}dx  \nonumber \\
&\leq -\alpha_1\int_\Omega udx+\frac{\beta_1\vert \Omega\vert }{b_1},
\end{align}
then (\ref{25}) follows from (\ref{27}) thanks to Gronwall's lemma.

To prove (\ref{26}), we first see that $\Vert v \Vert_{L^1(\Omega)}$ is bounded for $t\in(0,T_{\max})$ by the same arguments as above.  Taking $v$-equation in (\ref{1}), we have from the integration by parts that
\begin{align}\label{28}
 \frac{1}{p}\frac{d}{dt} \int_\Omega v^{p}(x,t)dx &=\int_\Omega v^{p-1}v_t dx=\int_\Omega v^{p-1}\Big(D_2\Delta v+(-\alpha_2+\frac{\beta_2}{a_2+b_2u+c_2v})v\Big) dx\nonumber \\ \nonumber
&=-\frac{4D_2(p-1)}{p^2} \int_\Omega \vert \nabla v^\frac{p}{2}\vert^2 dx+\int_\Omega (-\alpha_2+\frac{\beta_2}{a_2+b_2u+c_2v})v^{p} dx  \\ \nonumber
&\leq -\frac{4D_2(p-1)}{p^2} \int_\Omega \vert \nabla v^\frac{p}{2}\vert^2 dx-\alpha_2 \int_{\Omega} v^{p} dx+\frac{\beta_2}{c_2}\int_\Omega v^{p-1}dx \nonumber  \\
&\leq -\frac{4D_2(p-1)}{p^2} \int_\Omega \vert \nabla v^\frac{p}{2}\vert^2 dx -\frac{\alpha_2}{2}\int_{\Omega} v^{p} dx+C_3(p),
\end{align}
therefore, $\Vert v(\cdot,t) \Vert_{L^p(\Omega)}$ is bounded over $(0,T_{\max})$ for any $p\in(1,\infty)$.
\end{proof}

Lemma \ref{lemma23} provides the $L^\infty$--bound of $v$ and $L^1$--bound of $u$.  To establish the $L^\infty$--bound on $u$, we need to estimate $\Vert \nabla v \Vert_{L^p}$ for some large $p$.  For this purpose, we convert the $v$-equation in (\ref{1}) into the following abstract form
\begin{equation}\label{29}
v(\cdot,t)= e^{-D_2 At}v_0+\int_0^t e^{-D_2 A(t-s)} \big(D_2v(\cdot,s)+g(u(\cdot,s),v(\cdot,s) ) \big)ds,
\end{equation}
where $g(u,v)$ is given in (\ref{2}).  After applying the estimates (\ref{21})-(\ref{23}) on (\ref{29}), we have the following result.
\begin{lemma}\label{lemma24}
Let $(u,v)$ be a classical solution of (\ref{1})--(\ref{2}) over $\Omega \times (0,T_{\max})$.  For any $1\leq p \leq q \leq \infty$, there exists a positive constant $C$ dependent on $\Vert v_0 \Vert _{L^q(\Omega)}$ and $\Omega$ such that
\begin{equation}\label{210}
 \Vert v(\cdot,t) \Vert _{W^{1,q}(\Omega)} \leq C\left(1 + \int_0^te^{-\nu(t-s)} (t-s)^{-\frac{1}{2}-\frac{N}{2}(\frac{1}{p}-\frac{1}{q})}   \Vert v(\cdot,s)+1 \Vert_{L^p}   ds\right),
\end{equation}
 for any $t\in (0,T_{\max})$ provided that $q\in[1,\frac{Np}{N-p})$ if $p\in [1,N)$, $q\in [1,\infty)$ if $p=N$ and $q=\infty$ if $p>N$, where $\nu$ is the first Neumann eigenvalue of $-\Delta$ in $\Omega$.
\end{lemma}
By taking $p>N$ in (\ref{210}), we can quickly have the following result.
\begin{lemma}\label{lemma25}  Assume the same conditions as in Lemma \ref{lemma23}.  Then there exists a positive constant $C$
\begin{equation}\label{211}
\Vert   v(\cdot,t) \Vert_{W^{1,\infty}(\Omega)} \leq C, \forall t\in(0,T_{\max}).
\end{equation}
\end{lemma}

\subsection{Global existence of bounded classical solutions}

\begin{proof} [Proof\nopunct]\emph{of Theorem} \ref{theorem21}.
Thanks to (\ref{211}) in Lemma \ref{lemma25}, we only need to show that $\sup_{t\in(0,T_{\max})} \Vert u(\cdot,t) \Vert_{L^\infty}<\infty$, then we must have that $T_{\max}=\infty$; moreover, the regularities of solutions $(u,v)$ follow from Theorem \ref{theorem21}.

Without loss of our generality, we assume, in light of Lemma \ref{lemma25}, that $\Vert\phi(v(\cdot,t))\nabla v(\cdot,t) \Vert_{L^\infty(\Omega)}\leq1$ for all $t\in(0,T_{\max})$.  For any $p>1$, we test the first equation of (\ref{1}) by $pu^{p-1}$ and integrate it over $\Omega$ by parts to have that
\begin{align}\label{212}
&\frac{1}{p}\frac{d}{dt} \int_\Omega  u^p dx=\int_\Omega u^{p-1}u_t dx \nonumber \\
=&\int_\Omega u^{p-1}\Big(\nabla \cdot (D_1\nabla u+\chi u \phi(v)\nabla v)+(-\alpha_1+\frac{\beta_1}{a_1+b_1u+c_1v}\big)u\Big)   \nonumber \\
\leq & -\frac{4D_1(p-1)}{p^2} \int_\Omega\vert\nabla u^\frac{p}{2}\vert^2 dx +\frac{2(p-1)\chi}{p} \int_\Omega u^\frac{p}{2} \vert\nabla u^\frac{p}{2}\vert dx-\alpha_1\int_\Omega u^p dx+\frac{\beta_1}{b_1} \int_\Omega u^{p-1} dx \nonumber \\
\leq&-\frac{2D_1(p-1)}{p^2} \int_\Omega \vert\nabla u^\frac{p}{2}\vert^2 dx+\frac{(p-1)\chi^2}{2D_1}\int_\Omega  u^p dx+C_1(p),
\end{align}
where $C_1(p)=\frac{1}{q}\Big(\frac{\beta_1}{\alpha_1b_1p}\Big)^\frac{q}{p}\vert\Omega \vert$ with $q=\frac{p}{p-1}$ and we have applied the Young's inequality
\[ab\leq \epsilon a^p+\frac{b^q}{q(p\epsilon)^\frac{q}{p}},\forall a,b,\epsilon>0.\]
We recall Corollary 1 in \cite{CKWW} due to the Gagliardo--Nirenberg inequality: for any $\epsilon>0$, there exists a positive constant $C_0$ dependent on $N$ and $\Omega$ such that
\begin{equation}\label{213}
\Vert u^\frac{p}{2}\Vert^2_{L^2(\Omega)}\leq \epsilon\Vert \nabla u^\frac{p}{2} \Vert^2_{L^2(\Omega)}+C_0(1+\epsilon^{-\frac{N}{2}})\Vert u^\frac{p}{2} \Vert^2_{L^1(\Omega)}.
\end{equation}
Choosing $\epsilon=\big(\frac{\sqrt 2D_1}{p\chi}\big)^2$ in (\ref{213}), we see that (\ref{212}) becomes
\begin{equation}\label{214}
\frac{d}{dt} \int_\Omega  u^p dx \leq -\frac{p(p-1)\chi^2}{D_1}\int_\Omega  u^p dx+\frac{C_0p^2\chi^2}{2D_1^2}\Big(1+\big(\frac{p\chi}{\sqrt D_1}\big)^N\Big)\Big(\int_\Omega  u^{\frac{p}{2}} dx \Big)^2dx+C_1(p)p
\end{equation}
Finally, by applying the standard Moser--Alikakos iteration \cite{A0} to (\ref{214}), we can show that $\Vert u(\cdot,t) \Vert_{L^\infty(\Omega)}$ is uniformly bounded for $t\in(0,\infty)$.  This completes the proof of Theorem \ref{theorem21}.
\end{proof}

\section{Existence of nonconstant positive steady states}\label{section3}
In this section, we investigate nonconstant positive steady states of (\ref{1})--(\ref{2}) over one--dimensional finite domain in the following form
\begin{equation}\label{31}
\left\{
\begin{array}{ll}
(D_1 u'+\chi u \phi(v)v')'+\big(-1+\frac{1}{a_1+b_1u+c_1v}\big)u=0,&x \in (0,L),\\
D_2  v''+\big(-1+\frac{1}{a_2+b_2u+c_2v}\big)v=0,&x \in (0,L),\\
u'(x)=v'(x)=0,&x=0,L,
\end{array}
\right.
\end{equation}
where $'$ denotes the derivative taken with respect to $x$.  We have assumed in (\ref{31}) that $\alpha_i=\beta_i=1$, $i=1,2$ without loss of our generality.  Indeed, through the scalings
\[\tilde t=\alpha_1 t, \tilde x=\sqrt \alpha_1 x, \tilde c_1=\frac{\alpha_2c_1}{\beta_1}, \tilde c_2=\frac{\alpha_2^2c_2}{\alpha_1\beta_2}, \tilde \chi=\frac{\chi}{\alpha_1}\]
\[\tilde D_i=\frac{D_1}{\alpha_i}, \tilde a_i=\frac{a_i\alpha_i}{\beta_i}, i=1,2.\]
and $\tilde v=\frac{\alpha_1}{\alpha_2}v$, after dropping the \emph{tilde}s (\ref{1})--(\ref{2}) becomes
\begin{equation}\label{32}
\left\{
\begin{array}{ll}
u_t=\nabla\cdot(D_1\nabla u+\chi u\phi(v)\nabla v)+\big(-1+\frac{1}{a_1+b_1u+c_1v}\big)u,&x \in \Omega,t>0,\\
v_t=D_2\Delta v+\big(-1+\frac{1}{a_2+b_2u+c_2v}\big)v,&x \in \Omega,t>0,\\
\partial_\textbf{n}u=\partial_\textbf{n}v=0,&x\in\partial \Omega,t>0,\\
u(x,0)=u_0(x), v(x,0)=v_0(x),&x\in \Omega,
\end{array}
\right.
\end{equation}
Then the one--dimensional steady state of (\ref{32}) over $\Omega=(0,L)$ leads us to (\ref{31}).

We will see that large advection rate $\chi$ drives the emergence of nonconstant positive solutions to (\ref{31}).  There are four constant solutions to system (\ref{31}): $(0,0)$, $(0,\frac{1-a_2}{c_2})$, $(\frac{1-a_1}{b_1},0)$ and $(\bar u,\bar v)$, where
\begin{equation}\label{33}
\bar u=\frac{(1-a_2)c_1-(1-a_1)c_2}{b_2c_1-b_1c_2}, \bar v=\frac{(1-a_1)b_2-(1-a_2)b_1}{b_2c_1-b_1c_2},
\end{equation}
is the unique positive solution provided that
\begin{equation}\label{34}
\frac{c_1}{c_2}<\frac{1-a_1}{1-a_2}<\frac{b_1}{b_2},\quad\text{or}\quad  \frac{b_1}{b_2}<\frac{1-a_1}{1-a_2}<\frac{c_1}{c_2}.
\end{equation}
We assume this condition throughout the rest of our paper.  The first inequality in (\ref{34}) implies that $b_2c_1<b_1c_2$, therefore the inter--specific competition is weak compared to the intra-specific competition.  We call this condition the weak competition case.  Similarly, we call the latter condition in (\ref{34}) the strong competition case.  The same weak and strong competition cases are proposed in the studies of SKT competition models in \cite{SKT}.

\subsection{Diffusive system without advection}
We shall show that the emergence of nonconstant positive solutions to (\ref{31}) is driven by large advection rate $\chi$.   To see this, we study the existence of nonconstant positive solutions to (\ref{32}) with $\chi=0$, \emph{i.e.}, the following diffusive system
\begin{equation}\label{35}
\left\{
\begin{array}{ll}
u_t=D_1 u''+\big(-1+\frac{1}{a_1+b_1u+c_1v}\big)u,&x \in (0,L),t>0,\\
v_t=D_2 v''+\big(-1+\frac{1}{a_2+b_2u+c_2v}\big)v,&x \in (0,L),t>0,\\
u'(x)=v'(x)=0,&x=0,L, t>0,\\
u(x,0)=u_0(x), v(x,0)=v_0(x),&x\in (0,L).
\end{array}
\right.
\end{equation}
Our main result states as follows.
\begin{theorem}\label{theorem31}
Suppose that (\ref{34}) holds.  Then $(\bar u,\bar v)$ in (\ref{33}) is asymptotically stable with respect to (\ref{35}) and it is the only positive steady state of (\ref{35}) provided with either (i). $\frac{b_1}{b_2}>\frac{c_1}{c_2}$ or (ii). $\frac{b_1}{b_2}\leq\frac{c_1}{c_2}$ and $\max \{D_1,D_2\}$ is large.
\end{theorem}
\begin{proof}
Denote
\[f_0(u,v)=-1+\frac{1}{a_1+b_1u+c_1v},g_0(u,v)=-1+\frac{1}{a_2+b_2u+c_2v},\]
and we rewrite (\ref{35}) as
\begin{equation}
\left\{
\begin{array}{ll}
u_t=D_1 u''+f_0(u,v)u,&x \in (0,L),t>0\\
v_t=D_2 v''+g_0(u,v)v,&x \in (0,L),t>0\\
u'(x)=v'(x)=0,&x=0,L, t>0,\\
u(x,0)=u_0(x), v(x,0)=v_0(x),&x\in (0,L).
\end{array}
\right.
\end{equation}
Then the matrix
\begin{equation*}
\mathcal{M}=\Big(\frac{\partial (f_0,g_0)}{\partial (u,v)}\Big)\Big\vert_{(\bar u,\bar v)}=
\begin{pmatrix}
\frac{\partial f_0}{\partial u} & \frac{\partial f_0}{\partial v} \\
\frac{\partial g_0}{\partial u} & \frac{\partial g_0}{\partial v}
\end{pmatrix}\Big\vert_{(\bar u,\bar v)}
=\begin{pmatrix}
-b_1 & -c_1 \\
-b_2 & -c_2
\end{pmatrix},
\end{equation*}
has determinant $\vert \mathcal{M}\vert=(b_1c_2-b_2c_1)$.  According to Theorem 3.1 of \cite{LN}, $(\bar u,\bar v)$ is the only steady state of (\ref{35}) if either \emph{(i)} or \emph{(ii)} holds.

On the other hand, it follows from straightforward calculations that the linearized stability matrix corresponding to system (\ref{35}) at $(\bar u,\bar v)$ is
\begin{equation*}
\begin{pmatrix}
-D_1\big(\frac{k\pi}{L}\big)^2-b_1\bar u  & -c_1\bar u \\
-b_2\bar v& -D_2\big(\frac{k\pi}{L}\big)^2-c_2\bar v
\end{pmatrix},
\end{equation*}
which has two negative eigenvalues if either \emph{(i)} or \emph{(ii)} occurs, then by the same analysis in Theorem 2.5 in \cite{MOW} or \cite{S}, one can show that system (\ref{35}) generates a strongly monotone semi--flow on $C([0,L],\mathbb{R}^2)$ with respect to $\{(u,v)\in C([0,L],\mathbb{R}^2) \vert u>0,v>0 \}$, hence $(\bar u,\bar v)$ is globally asymptotically stable.  This completes the proof of Theorem \ref{theorem31}.
\end{proof}

Our results indicate that the global dynamics of the diffusion system (\ref{35}) is dominated by the ODEs in the weak competition case and also in the strong case if one of the diffusion rates is large.  However, if both $D_1$ and $D_2$ are small, $(\bar u,\bar v)$ is unstable in the strong competition case.  We surmise that positive solutions with nontrivial patterns may arise when the system parameters and the domain geometry are properly balanced.  Nonconstant positive solutions with spikes are investigated for the diffusion system (\ref{31}) with Lotka--Volterra dynamics by various authors.  See \cite{KY,MM,M1,MEF}.

\subsection{Advection--driven instability}
Theorem \ref{31} states that diffusion does not change the dynamics of the spatially homogeneous solution and no Turing's instability occurs for the diffusive system (\ref{35}) in most cases.  We proceed to investigate the effect of advection $\chi$ on the emergence of nonconstant positive solutions to (\ref{31}).  We shall show that this equilibrium loses its stability as the advection rate $\chi$ crosses a threshold value.  To this end, we first study the linearized stability of the equilibrium $(\bar u,\bar v)$.  Let $(u,v)=(\bar u,\bar v)+(U,V)$, where $U$ and $V$ are small perturbations from $(\bar u,\bar v)$, then
\begin{equation}\label{37}
\left\{
\begin{array}{ll}
U_t\approx D_1 U^{''}+\chi \bar u \phi(\bar v) V^{''}-b_1\bar uU-c_1\bar u V,&x \in (0,L),t>0,\\
V_t\approx D_2V^{''}-b_2\bar vU-c_2\bar vV,&x \in (0,L),t>0,\\
U'(x)=V'(x)=0,&x=0,L,t>0.
\end{array}
\right.
\end{equation}
We have the following result on the linearized instability of $(\bar u,\bar v)$ to (\ref{31}).
\begin{proposition}\label{proposition1}
The constant solution $(\bar u,\bar v)$ of (\ref{31}) is unstable if
\begin{equation}\label{38}
\chi>\chi_0=\min_{k \in \mathbb{N^+} } \frac{\big(D_1(\frac{k\pi}{L})^2+b_1\bar u\big)\big(D_2(\frac{k\pi}{L})^2+c_2\bar v\big)-b_2c_1\bar u\bar v}{b_2(\frac{k\pi}{L})^2\phi(\bar v)\bar u\bar v}.
\end{equation}
\end{proposition}
\begin{proof}
According to the standard linearized stability analysis, the stability of $(\bar u,\bar v)$ is determined by the eigenvalues of the following matrix
\begin{equation}\label{39}
 \mathcal H_k=
\begin{pmatrix} -D_1(\frac{k\pi}{L})^2-b_1\bar u& -\chi\bar u\phi(\bar v)(\frac{k\pi}{L})^2-c_1\bar u\\
-b_2 \bar v &-D_2(\frac{k\pi}{L})^2-c_2\bar v \end{pmatrix}, k\in \mathbb N^+.
\end{equation}
In particular, $(\bar u,\bar v)$ is unstable if $\mathcal H_k$ has an eigenvalue with positive real part for some $k\in \mathbb{N^+}$. It is easy to see that the characteristic polynomial of (\ref{39}) takes the form
\[p_k(\lambda)=\lambda^2+T_k\lambda+D_k,\]
where \[T_k=(D_1+D_2)(\frac{k\pi}{L})^2+b_1\bar u+c_2\bar v>0,\]
and\[D_k=\big(D_1(\frac{k\pi}{L})^2+b_1\bar u\big)\big(D_2(\frac{k\pi}{L})^2+c_2\bar v\big)-\big(\chi\bar u\phi(\bar v)(\frac{k\pi}{L})^2+c_1\bar u\big)b_2\bar v.\]
$p_k(\lambda)$ has a positive root if and only if $p_k(0)=D_k<0$, then (\ref{38}) follows from simple calculations and the proof completes.
\end{proof}
We have from Proposition \ref{proposition1} that $(\bar u,\bar v)$ loses its stability when $\chi$ surpasses $\chi_0$.  In the weak competition case $\frac{b_1}{b_2}>\frac{1-a_1}{1-a_2}>\frac{c_1}{c_2}$, we see that $\chi_0$ is always positive, hence $(\bar u,\bar v)$ remains locally stable for $\chi$ being small.  In the strong competition case $\frac{b_1}{b_2}<\frac{1-a_1}{1-a_2}<\frac{c_1}{c_2}$, $\chi_0<0$ if both $D_1$ and $D_2$ are sufficiently small.  This corresponds to the fact that $(\bar u,\bar v)$ is unstable for $\chi=0$ in (\ref{31}) in this case.  By the same stability analysis above, we can show that the appearance of advection rate $\chi$ does not change the stability of the rest equilibrium points.  It is also worthwhile to point out that Proposition \ref{proposition1} carries over to higher dimensions with $(\frac{k\pi}{L})^2$ replaced by the $k$--eigenvalue of Neumann Laplacian.

\subsection{Steady state bifurcation}
To establish the existence of nonconstant positive solutions to (\ref{31}), we shall use the bifurcation theory due to Crandall--Rabinowitz \cite{CR} by taking $\chi$ as the bifurcation parameter.  To this end, we rewrite (\ref{31}) into the following abstract form
\[\mathcal{F}(u,v,\chi)=0,~(u,v,\chi) \in \mathcal{X}  \times \mathcal{X} \times \mathbb{R},\]
where
\begin{equation}\label{310}
\mathcal{F}(u,v,\chi) =\left(
 \begin{array}{c}
(D_1 u'+\chi u \phi(v)v')'+(-1+\frac{1}{a_1+b_1u+c_1 v})u\\
~~\\
D_2v^{''}+(-1+\frac{1}{a_2+b_2u+c_2 v})v
 \end{array}
 \right),
 \end{equation}
and $\mathcal X=\{H^2(0,L)\vert u'(0)=u'(L)=0\}$.  We collect some facts about $\mathcal{F}$.  First of all, $\mathcal{F}(\bar u,\bar v,\chi)=0$ for any $\chi\in \mathbb{R}$ and $\mathcal{F}:\mathcal{X} \times \mathbb{R}\times \mathbb{R}\rightarrow \mathcal{Y}\times\mathcal{Y}$ is analytic for $\mathcal{Y}=L^2(0,L)$.  For any fixed $(u_0,v_0) \in \mathcal{X}\times\mathcal{X}$, the Fr\'echet derivative of $\mathcal{F}$ is
\begin{equation}\label{311}
D_{(u,v)}\mathcal{F}(u_0,v_0,\chi)(u,v)=\left(
\begin{array}{c}
D_1u^{''}+\chi\big((\phi(v_0)u+u_0\phi^{'}(v_0)v)v_0^{'}+u_0\phi(v_0)v'\big)^{'}+R\\
~~\\
D_2 v^{''}-\frac{b_2v_0}{(a_2+b_2u_0+c_2v_0)^2}u+\big(-1+\frac{a_2+b_2u_0}{(a_2+b_2u_0+c_2v_0)^2}\big)v
\end{array}
\right),
\end{equation}
where $R=\big(-1+\frac{a_1+c_1v_0}{(a_1+b_1u_0+c_1v_0)^2}\big)u-\frac{c_1u_0}{(a_1+b_1u_0+c_1v_0)^2}v$.
By the same arguments that lead to (iv) of Lemma 5.1 in \cite{CKWW} or Lemma 2.3 of \cite{WX}, one can show that $D_{(u,v)}\mathcal{F}(u_0,v_0,\chi)$ is Fredholm with zero index.

For bifurcation to occur at $(\bar u, \bar v,\chi)$, we need the Implicit Function Theorem to fail on $\mathcal{F}$ at this point, hence we require the following necessary condition $\mathcal{N}\big(D_{(u,v)}\mathcal{F}(\bar u,\bar v,\chi)\big)\neq {0}$.  Let $(u,v)$ be a nontrivial solution in this null--space, then it satisfies the following system
\begin{equation}\label{312}
\left\{
\begin{array}{ll}
D_1 u^{''}+\chi\bar u\phi(\bar v)v^{''}-b_1\bar u u-c_1\bar uv=0,&x \in (0,L),\\
D_2  v''-b_2\bar vu-c_2\bar vv=0,&x \in (0,L),\\
u'(x)=v'(x)=0,&x=0,L.
\end{array}
\right.
\end{equation}
Expanding $u$ and $v$ into the following series
\[u(x)=\sum_{k=0}^{\infty}t_k\cos\frac{k\pi x}{L}, v(x)=\sum_{k=0}^{\infty}s_k\cos\frac{k\pi x}{L},\]
and substituting them into (\ref{312}) yield
\begin{equation}\label{313}
\begin{pmatrix}
-D_1(\frac{k\pi}{L})^2-b_1\bar u & -\chi\bar u\phi(\bar v)(\frac{k\pi}{L})^2-c_1\bar u  \\
~~\\
-b_2\bar v & -D_2(\frac{k\pi}{L})^2-c_2\bar v
\end{pmatrix}
\begin{pmatrix}
t_k\\
~~\\
s_k
\end{pmatrix}\!\!=\!\!\begin{pmatrix}
\!0\\
~~\\
0
\end{pmatrix}.
\end{equation}
$k=0$ is ruled out in (\ref{312}) thanks to (\ref{34}).  For $k\in \mathbb{N^+}$, (\ref{313}) has nonzero solutions if and only if its coefficient matrix is singular which implies that local bifurcation might occur at
\begin{equation}\label{314}
\chi=\chi_k=\frac{\big(D_1(\frac{k\pi}{L})^2+b_1\bar u\big)\big(D_2(\frac{k\pi}{L})^2+c_2\bar v\big)-b_2c_1\bar u\bar v}{b_2(\frac{k\pi}{L})^2\phi(\bar v)\bar u\bar v}, k\in \mathbb N^+.
\end{equation}
Moreover, the null space is one--dimensional and has a span
\[\mathcal{N}\big(D_{(u,v)}\mathcal{F}(\bar u,\bar v,\chi_k)\big)=\text{span}\{(\bar u_k,\bar v_k)\},\]
where
\begin{equation}\label{315}
(\bar u_k,\bar v_k)=(Q_k,1)\cos\frac{k\pi x}{L}, Q_k=-\frac{D_2\big(\frac{k\pi}{L}\big)^2+c_2\bar v}{b_2\bar v}.
\end{equation}
Having the candidates for bifurcation values $\chi_k$, we now show that local bifurcation does occur at $(\bar u,\bar v,\chi_k)$ in the following theorem, which establishes nonconstant positive solutions to (\ref{31}).

\begin{theorem}\label{theorem32}
Suppose that $\phi\in C^2(\mathbb{R},\mathbb{R})$ and $\phi(v)>0$ for all $v>0$. Assume that (\ref{34}) holds, and for all positive different integers $k,j\in\mathbb{N^+}$,
\begin{equation}\label{316}
(b_1c_2-b_2c_1)\bar u\bar v\neq k^2j^2D_1D_2(\frac{\pi}{L})^4, k\neq j,
\end{equation}
where $(\bar u,\bar v)$ is the positive equilibrium of (\ref{31}) given in (\ref{33}).  Then for each $k\in\mathbb{N}^+$, there exists $\delta>0$ small such that (\ref{31}) admits nonconstant solutions
$\big(u_k(s,x),v_k(s,x),\chi_k(s)\big)\in\mathcal{X}\times \mathcal{X}\times \mathbb{R^+}$ with $\big(u_k(0,x),v_k(0,x), \chi_k(0)\big)=(\bar u,\bar v,\chi_k)$.  The solutions are continuous functions of $s$ in the topology of $\mathcal X\times \mathcal X\times \mathbb R$ and have the following expansions for $\vert s\vert$ being small,
\begin{equation}\label{317}
(u_k(s,x),v_k(s,x))=(\bar u,\bar v)+s(Q_k,1)\cos\frac{k\pi x}{L}+o(s),
\end{equation}
and $\big(u_k(s,x),v_k(s,x)\big)-(\bar u,\bar v)-s(Q_k,1)\cos\frac{k\pi x}{L}\in \mathcal{Z}$ where
\begin{equation}\label{318}
\mathcal{Z}=\big\{(u,v)\in\mathcal{X}\times \mathcal{X}\big\vert\int_0^L u\bar u_k+v\bar v_k dx=0\big\}.
\end{equation}
with $(\bar u_k,\bar v_k)$ and $Q_k$ defined in (\ref{315}); moreover, all nontrivial solutions of (\ref{31}) around $(\bar u,\bar v,\chi_k)$ must stay on the curve $\Gamma_k(s)=\big(u_k(s,x),v_k(s,x),\chi_k(s)\big)$, $s\in (-\delta,\delta)$.
\end{theorem}

\begin{proof}
Our results follow from Theorem 1.7 of Crandall and Rabinowitz \cite{CR} once we prove the following transversality condition,
\begin{equation}\label{319}
\frac{d}{d \chi} \left(D_{(u,v)}\mathcal{F}(\bar{u},\bar{v},\chi)\right)(\bar u_k,\bar v_k)\vert_{\chi=\chi_k} \notin \mathcal{R}(D_{(u,v)}\mathcal{F}(\bar{u},\bar{v},\chi_k)),
\end{equation}
where $(\bar u_k,\bar v_k)$ is given in (\ref{315}) and $\mathcal{R}(\cdot)$ denotes range of the operator.  We argue by contradiction and assume that (\ref{319}) fails.  Therefore there exist a nontrivial pair $(u,v)$ to the following problem
\begin{equation}\label{320}
\left\{\!\!
\begin{array}{ll}
D_1 u^{''}+\chi_k\bar{u}\phi(\bar{v}) v^{''}-b_1\bar uu-c_1\bar uv
=- (\frac{k\pi}{L})^2\bar u\phi(\bar v)\cos\frac{k\pi x}{L} ,&x\in(0,L),\\
D_2  v''-b_2\bar{v}u-c_2\bar{v}v=0,&x\in(0,L),\\
u'(x)=v'(x)=0,&x=0,L.
\end{array}
\right.
\end{equation}
Testing the first two equations in (\ref{320}) by $\cos\frac{k\pi x}{L}$ over $(0,L)$ yields
\begin{equation}\label{321}
\begin{pmatrix} -D_1(\frac{k\pi}{L})^2-b_1\bar u& -\chi_k\bar u\phi(\bar v)(\frac{k\pi}{L})^2-c_1\bar u\\
~~\\
-b_2\bar v&-D_2(\frac{k\pi}{L})^2-c_2\bar v \end{pmatrix}\begin{pmatrix}
\int_0^L u\cos\frac{k\pi x}{L}dx\\
~~\\
\int_0^L v\cos\frac{k\pi x}{L}dx
\end{pmatrix}\!\!=\!\!\begin{pmatrix}
\!-\frac{(k\pi)^2\bar u\phi(\bar v)}{2L}\\
~~\\
0
\end{pmatrix}\!\!
\end{equation}
The coefficient matrix is singular due to (\ref{314}), then we reach a contradiction in (\ref{321}) and this proves (\ref{319}).  On the other hand, we need condition (\ref{316}) such that $\chi_k\neq\chi_j$ for all integers $k\neq j$, which is also required for the application of the local theory in \cite{CR}.  The proof of Theorem \ref{theorem32} is complete.
\end{proof}
\subsection{Stability analysis of the nonconstant steady states}

We proceed to study stabilities of the nontrivial bifurcating solutions $(u_k(s,x),v_k(s,x),\chi_k(s))$ obtained in Theorem \ref{theorem32}.  Here the stability or instability means that of the nonconstant solution viewed as an equilibrium of the time-dependent system of (\ref{31}).  $\mathcal F$ is $C^4$-smooth in $s$ if $\phi$ is $C^4$, therefore, according to Theorem 1.18 in \cite{CR}, $(u_k(s,x),v_k(s,x),\chi_k(s))$ is $C^3$-smooth, we have the following expansions,
\begin{equation}\label{322}
\left\{
\begin{array}{ll}
u_k(s,x)=\bar u+sQ_k\cos\frac{k\pi x}{L}+s^2\psi_1(x)+s^3\psi_2(x)+o(s^3),\\
v_k(s,x)=\bar v+s\cos\frac{k\pi x}{L}+s^2\varphi_1(x)+s^3\varphi_2(x)+o(s^3),\\
\chi_k(s)=\chi_k+sK_1+s^2K_2+o(s^2).
\end{array}
\right.
\end{equation}
where $(\psi_i,\varphi_i)\in \mathcal{Z}$ as defined in (\ref{318}) and $K_i$ is a constant for $i=1,2$.  Moreover, we have from Taylor's Theorem that
\begin{equation}\label{323}
\phi\big(v_k(s,x)\big)=\phi(\bar v)+s\phi^{'}(\bar v)\cos\frac{k\pi x}{L}+s^2\Big(\phi^{'}(\bar v)\varphi_1+\frac{1}{2}\phi^{''}(\bar v)\cos^2\frac{k\pi x}{L}\Big)+o(s^3).
\end{equation}
$o(s^3)$-terms in (\ref{322}) and (\ref{323}) are taken in $H^2$-topology.  For the sake of simplicity, we introduce the following notations
\begin{equation}\label{324}
\begin{split}
P_1=&\frac{1}{2}\Big(f_{uu}(\bar u,\bar v)Q_k^2+2f_{uv}(\bar u,\bar v)Q_k+f_{vv}(\bar u,\bar v)\Big)\\
   =&b_1(b_1\bar u-1)Q_k^2+c_1(2b_1\bar u-1)Q_k+c_1^2\bar u  \\
P_2=&\frac{1}{2}\Big(g_{uu}(\bar u,\bar v)Q_k^2+2g_{uv}(\bar u,\bar v)Q_k+g_{vv}(\bar u,\bar v)\Big)\\
   =&b_2^2\bar vQ_k^2+b_2(2c_2\bar v-1)Q_k+c_2(c_2\bar v-1)  \\
P_3=&\frac{1}{6}\Big(f_{uuu}(\bar u,\bar v)Q_k^3+3f_{uuv}(\bar u,\bar v)Q_k^2+3f_{uvv}(\bar u,\bar v)Q_k+f_{vvv}(\bar u,\bar v)\Big) \\
   =&b_1^2(1-b_1\bar u)Q_k^3+b_1c_1(2-3b_1\bar u)Q_k^2+c_1^2(1-3b_1\bar u)Q_k-c_1^3\bar u\\
P_4=&\frac{1}{6}\Big(g_{uuu}(\bar u,\bar v)Q_k^3+3g_{uuv}(\bar u,\bar v)Q_k^2+3g_{uvv}(\bar u,\bar v)Q_k+g_{vvv}(\bar u,\bar v)\Big) \\
   =&-b_2^3\bar vQ_k^3+b_2^2(1-3c_2\bar v)Q_k^2+b_2c_2(2-3c_2\bar v)Q_k+c_2^2(1-c_2\bar v)\\
P_5=&f_{uu}(\bar u,\bar v)Q_k+f_{uv}(\bar u,\bar v)=2b_1(b_1\bar u-1)Q_k+c_1(2b_1\bar u-1)\\
P_6=&f_{uv}(\bar u,\bar v)Q_k+f_{vv}(\bar u,\bar v)=c_1(2b_1\bar u-1)Q_k+2c_1^2\bar u\\
P_7=&g_{uu}(\bar u,\bar v)Q_k+g_{uv}(\bar u,\bar v)=2b_2^2\bar vQ_k+b_2(2c_2\bar v-1)\\
P_8=&g_{uv}(\bar u,\bar v)Q_k+g_{vv}(\bar u,\bar v)=b_2(2c_2\bar v-1)Q_k+2c_2(c_2\bar v-1).
\end{split}
\end{equation}

Substituting (\ref{322}) and (\ref{323}) into (\ref{31}) and equating the $s^2$-terms, we collect
\begin{equation}\label{325}
\left\{
\begin{array}{ll}
D_1\psi_1^{''}+\chi_k\bar u\phi(\bar v)\varphi_1^{''}-b_1\bar u\psi_1-c_1\bar u\varphi_1-K_1\bar u\phi(\bar v)(\frac{k\pi}{L})^2\cos\frac{k\pi x}{L} \\ =\chi_k(\frac{k\pi}{L})^2\big(\phi(\bar v)Q_k+\bar u\phi^{'}(\bar v)\big)\cos\frac{2k\pi x}{L}-P_1\cos^2\frac{k\pi x}{L},\\
D_2\varphi_1^{''}-b_2\bar v\psi_1-c_2\bar v\varphi_1=-P_2\cos^2\frac{k\pi x}{L},\\
\psi'_1(0)=\phi'_1(0)=0, \psi'_1(L)=\phi'_1(L)=0.
\end{array}
\right.
\end{equation}
Multiplying the first equation of (\ref{325}) by $\cos\frac{k\pi x}{L}$ and integrating it over $(0,L)$ by parts yield
\begin{equation}\label{326}
\begin{split}
\frac{(k\pi)^2\bar u\phi(\bar v)K_1}{2L}=&-\Big(\chi_k\bar u\phi(\bar v)(\frac{k\pi }{L})^2+c_1\bar u\Big)\int_0^L\varphi_1\cos\frac{k\pi x}{L}dx   \\
&-\Big(D_1(\frac{k\pi }{L})^2+b_1\bar u\Big)\int_0^L\psi_1\cos\frac{k\pi x}{L}dx.
\end{split}
\end{equation}
Multiplying the second equation by $\cos\frac{k\pi x}{L}$ and integrating it over $(0,L)$ by parts yield
\begin{equation}\label{327}
b_2\bar v\int_0^L\psi_1\cos\frac{k\pi x}{L}dx+\Big(D_2(\frac{k\pi}{L})^2+c_2\bar v\Big)\int_0^L\varphi_1\cos\frac{k\pi x}{L}dx=0.
\end{equation}
On the other hand, (\ref{318}) and the fact that $(\psi_1,\varphi_1)\in \mathcal{Z}$ give us
\begin{equation}\label{328}
Q_k\int_0^L\psi_1\cos\frac{k\pi x}{L}dx+\int_0^L\varphi_1\cos\frac{k\pi x}{L}dx=0,
\end{equation}
where $Q_k=-\frac{D_2(\frac{k\pi}{L})^2+c_2\bar v}{b_2\bar v}$. Solving (\ref{327}) and (\ref{328}) leads us to
\[(1+Q_k^2)\int_0^L\psi_1\cos\frac{k\pi x}{L}dx=0,\]
which implies that
\[\int_0^L\psi_1\cos\frac{k\pi x}{L}dx=\int_0^L\varphi_1\cos\frac{k\pi x}{L}dx=0,~ \forall k\in\mathbb{N}^+.\]
It follows  from (\ref{326}) that $K_1=0$, hence the bifurcation branch $\Gamma_k(s)$ around $\chi_k$ is of pitch-fork type, \emph{\emph{i.e.}}, one--sided.  We proceed to evaluate $K_2$ which determines branch direction hence the stability of $(u_k(s,x),v_k(s,x),\chi_k(s))$ as we shall see in the coming analysis.

Equating the $s^3$-terms in (\ref{31}), we collect
\begin{equation}\label{329}
\left\{
\begin{array}{ll}
\begin{split}
D_1\psi_2^{''}+\chi_k\bar u\phi(\bar v)\varphi_2^{''}-b_1\bar u\psi_2-c_1\bar u\varphi_2-K_2\bar u\phi(\bar v)(\frac{k\pi}{L})^2\cos\frac{k\pi x}{L}\\
=\chi_k A_0-P_5\psi_1\cos\frac{k\pi x}{L}-P_6\varphi_1\cos\frac{k\pi x}{L}-P_3\cos^3\frac{k\pi x}{L},
\end{split}\\
\begin{split}
D_2\varphi_2^{''}-b_2\bar v\psi_2-c_2\bar v\varphi_2=-P_7\psi_1\cos\frac{k\pi x}{L}-P_8\varphi_1\cos\frac{k\pi x}{L}-P_4\cos^3\frac{k\pi x}{L},
\end{split}    \\
\psi'_2 (0)=\varphi_2'(0)=0, \psi'_2 (L)=\varphi_2'(L)=0,
\end{array}
\right.
\end{equation}
where
\[\begin{split}
A_0=&-(\phi(\bar v)Q_k+\bar u\phi^{'}(\bar v))\varphi_1^{''}\cos\frac{k\pi x}{L}+(\frac{k\pi}{L})^2(\phi(\bar v)\psi_1+\bar u\phi^{'}(\bar v)\varphi_1)\cos\frac{k\pi x}{L}\\
&+(\frac{k\pi}{L})\big(\phi(\bar v)\psi_1^{'}+(\phi(\bar v)Q_k+2\bar u\phi^{'}(\bar v))\varphi_1^{'}\big)\sin\frac{k\pi x}{L}\\
&-(\frac{k\pi}{L})^2\big(2\phi^{'}(\bar v)Q_k+\bar u\phi^{''}(\bar v)\big)\cos\frac{k\pi x}{L}\\
&+3(\frac{k\pi}{L})^2\big(\phi^{'}(\bar v)Q_k+\frac{1}{2}\bar u\phi^{''}(\bar v)\big)\cos^3\frac{k\pi x}{L}.
\end{split}\]
Testing the first equation in (\ref{329}) by $\cos\frac{k\pi x}{L}$ over $(0,L)$, we conclude from straightforward calculations that
\begin{equation}\label{330}
\begin{split}
\frac{\bar u\phi(\bar v)(k\pi)^2}{L}K_2=&-2\Big(D_1(\frac{k\pi}{L})^2+b_1\bar u\Big)\int_0^L\psi_2\cos\frac{k\pi x}{L}dx  \\
&-2\Big(\chi_k\bar u\phi(\bar v)(\frac{k\pi}{L})^2 +c_1\bar u\Big)\int_0^L\varphi_2\cos\frac{k\pi x}{L}dx   \\
&+\Big(P_5-\chi_k\phi(\bar v)(\frac{k\pi}{L})^2\Big)\int_0^L\psi_1dx  \\
&+\Big(P_6-\chi_k\bar u\phi^{'}(\bar v)(\frac{k\pi}{L})^2\Big)\int_0^L\varphi_1dx  \\
&+\Big(P_5+\chi_k\phi(\bar v)(\frac{k\pi}{L})^2\Big)\int_0^L\psi_1\cos\frac{2k\pi x}{L}dx \\
&+\Big(P_6-\chi_k\big(2\phi(\bar v)Q_k+\bar u\phi^{'}(\bar v)\big)(\frac{k\pi}{L})^2\Big)\int_0^L\varphi_1\cos\frac{2k\pi x}{L}dx  \\
&+\frac{3LP_3}{4}-\Big(\phi^{'}(\bar v)Q_k+\frac{1}{2}\bar u\phi^{''}(\bar v)\Big)\frac{k^2\pi^2\chi_k}{4L},
\end{split}
\end{equation}
where $P_3$, $P_5$ and $P_6$ are given in (\ref{324}), and $\chi_k$ is defined by (\ref{314}).

On the other hand, testing the second equation of (\ref{329}) by $\cos\frac{k\pi x}{L}$ over $(0,L)$ gives rise to
\begin{equation}\label{331}
\begin{split}
&b_2\bar v\int_0^L\psi_2\cos\frac{k\pi x}{L}dx +\Big(D_2(\frac{k\pi}{L})^2+c_2\bar v\Big)\int_0^L \varphi_2\cos\frac{k\pi x}{L}dx \\
=&\frac{P_7}{2}\Big(\int_0^L\psi_1dx+\int_0^L\psi_1\cos\frac{k\pi x}{L}dx\Big)+\frac{P_8}{2}\Big(\int_0^L\varphi_1dx+\int_0^L\varphi_1\cos\frac{k\pi x}{L}dx\Big)-\frac{3P_4L}{8},
\end{split}
\end{equation}
Since $(\psi_2,\varphi_2)\in \mathcal{Z}$, we can solve (\ref{331}) to find that
\begin{equation}\label{332}
\begin{split}
\int_0^L \psi_2\cos\frac{k\pi x}{L}dx=&\frac{(\int_0^L\psi_1\cos\frac{2k\pi x}{L}dx+\int_0^L\psi_1dx\big)P_7}{2b_2\bar v(Q_k^2+1)}  \\
&+\frac{(\int_0^L\varphi_1\cos\frac{2k\pi x}{L}dx+\int_0^L\varphi_1dx\big)P_8}{2b_2\bar v(Q_k^2+1)}  \\
&+\frac{3P_4L}{8b_2\bar v(Q_k^2+1)},
\end{split}
\end{equation}
\begin{equation}\label{333}
\begin{split}
\int_0^L \varphi_2\cos\frac{k\pi x}{L}dx=&-\frac{Q_k(\int_0^L\psi_1\cos\frac{2k\pi x}{L}dx+\int_0^L\psi_1dx\big)P_7}{2b_2\bar v(Q_k^2+1)}  \\
&-\frac{Q_k(\int_0^L\varphi_1\cos\frac{2k\pi x}{L}dx+\int_0^L\varphi_1dx\big)P_8}{2b_2\bar v(Q_k^2+1)}  \\
&-\frac{3Q_kP_4L}{8b_2\bar v(Q_k^2+1)}.
\end{split}
\end{equation}

In order to find $K_2$ in (\ref{330}), we see from (\ref{332}) and (\ref{333}) that it is necessary to evaluate the following integrals
\[\int_0^L \psi_1dx,~ \int_0^L \varphi_1dx,~\int_0^L \psi_1\cos\frac{2k\pi x}{L}dx,~\text{and} ~\int_0^L \varphi_1\cos\frac{2k\pi x}{L}dx.\]
Integrating both equations in (\ref{325}) over $(0,L)$, we obtain from straightforward calculations by taking the fact $K_1=0$ that
\begin{equation}\label{334}
\int_0^L\psi_1dx=\frac{L(c_2\bar vP_1-c_1\bar uP_2)}{2(b_1c_2-b_2c_1)\bar u\bar v}, \int_0^L\varphi_1dx=\frac{L(b_1\bar uP_2-b_2\bar vP_1)}{2(b_1c_2-b_2c_1)\bar u\bar v}
\end{equation}
On the other hand, we multiply both two equations in (\ref{325}) by $\cos\frac{2k\pi x}{L}$ and integrate them over $(0,L)$ by parts. Then again thanks to $K_1=0$, we have from straightforward calculations that
\begin{equation}\label{335}
\int_0^L\psi_1\cos\frac{2k\pi x}{L}dx=\frac{\vert\mathcal{A}_1\vert}{\vert\mathcal{A}\vert},~ \int_0^L\psi_1\cos\frac{2k\pi x}{L}dx=\frac{\vert\mathcal{A}_2\vert}{\vert\mathcal{A}\vert}.
\end{equation}
where
\[\vert\mathcal{A}\vert=12D_1D_2(\frac{k\pi}{L})^4-3(b_1c_2-b_2c_1)\bar u\bar v,\]
\[\begin{split}
\vert\mathcal{A}_1\vert=&-2\chi_kD_2L(\frac{k\pi}{L})^4\big(\phi(\bar v)Q_k+\bar u\phi^{'}(\bar v)\big)+\frac{L}{4}(c_2\bar vP_1-c_1\bar uP_2)\\
&+\frac{k^2\pi^2}{L}\Big(D_2P_1-\chi_k\bar u\phi(\bar v)P_2-\frac{1}{2}\chi_kc_2\bar v\big(\phi(\bar v)Q_k+\bar u\phi^{'}(\bar v)\big)\Big),
\end{split}\]
and \[\vert\mathcal{A}_2\vert=\frac{k^2\pi^2}{L}\Big(D_1P_2+\chi_kb_2\bar v\big(\phi(\bar v)Q_k+\bar u\phi^{'}(\bar v)\big)\Big) +\frac{L}{4}(b_1\bar uP_2-b_2\bar vP_1).\]
Finally, we are able to evaluate $K_2$ in terms of system parameters, thanks to (\ref{330}), (\ref{334}) and (\ref{335}).  The rest calculations are straightforward but tedious and we skip them here.  We present the stability of the bifurcating solutions in the following Theorem.
\begin{theorem}\label{theorem33}
Assume the same conditions as in Theorem \ref{theorem32}.  Let $K_2$ be given in (\ref{330}).  Suppose that $\chi_{k_0}=\min_{k\in\mathbb N^+}\chi_k$ of (\ref{311}).  Then for all positive integers $k\neq k_0$ and small $\delta>0$, $(u_k(s,x),v_k(s,x))$ is unstable for $s\in(-\delta,\delta)$; moreover, $(u_{k_0}(s,x),v_{k_0}(s,x))$ is stable for $s\in(-\delta,\delta)$ if $K_2>0$ and it is unstable for $s\in(-\delta,\delta)$ if $K_2<0$.
\end{theorem}
\begin{remark}
Theorem \ref{theorem33} provides a rigourous selection mechanism for stable nonconstant positive solutions to system (\ref{31}).  If $(u_k(s,x),v_k(s,x))$ is stable, then $k$ must be the integer that minimizes the bifurcation value $\chi_k$ over $\mathbb N^+$, that being said, if a bifurcation branch is stable, then it must be the first branch counting from the left to the right.
\end{remark}
\begin{proof}
For each $k\in \mathbb N^+$, we linearize (\ref{31}) around $(u_k(s,x),v_k(s,x),\chi_k(s))$ and obtain the following eigenvalue problem
\begin{equation}\label{336}
D_{(u,v)}\mathcal{F}(u_k(s,x),v_k(s,x),\chi_k(s))(u,v)=\lambda(s)(u,v),~(u,v)\in \mathcal{X} \times \mathcal{X}.
\end{equation}
Then solution $(u_k(s,x),v_k(s,x),\chi_k(s))$ will be asymptotically stable if and only if the real part of any eigenvalue $\lambda(s)$ to (\ref{336}) is negative for $s\in (-\delta,\delta)$.

Sending $s$ to 0, we know that $\lambda=\lambda(0)=0$ is a simple eigenvalue of $D_{(u,v)}\mathcal{F}(\bar{u},\bar{v},\chi_k)(u,v)=\lambda(u,v)$ or equivalently, the following eigenvalue problem
\begin{equation}\label{337}
\left\{
\begin{array}{ll}
D_1 u''+\chi_k \bar{u}\phi(\bar{v})v''-b_1\bar{u}u-c_1\bar{u}v=\lambda u,& x\in(0,L),  \\
D_2 v''-b_2\bar{v}u-c_2\bar{v}v=\lambda v,& x\in(0,L),  \\
u'(x)=v'(x)=0,&x=0,L;
\end{array}
\right.
\end{equation}
moreover, it has a one--dimensional eigen-space $\mathcal{N}\big(D_{(u,v)}\mathcal{F}(\bar{u},\bar{v},\chi_k)\big)=\{(Q_k,1)\cos \frac{k\pi x}{L}\}$ and one can also prove that $(Q_k,1)\cos \frac{k\pi x}{L} \not\in\mathcal{R}\big(D_{(u,v)}\mathcal{F}(\bar{u},\bar{v},\chi_k)\big)$ following the same analysis that leads to (\ref{319}).

Multiplying (\ref{337}) by $\cos\frac{k\pi x}{L}$ and integrating it over $(0,L)$ by parts give rise to
\[\begin{pmatrix}\!\!
-D_1 \!\big(\frac{k \pi}{L}\big)^2\!\!-\!b_1\bar{u}-\lambda & -\chi_k \big(\frac{k \pi}{L}\big)^2 \!\bar{u}\phi(\bar{v})\!-\!c_1\bar{u}   \\
-b_2\bar{v} & -D_2\big(\frac{k \pi}{L}\big)^2\!\!-\!c_2\bar{v}-\lambda
\!\!\end{pmatrix}\!
\!\!\begin{pmatrix}\!
\int_0^L \!\!{u}\cos\frac{k\pi x}{L}dx \\
\int_0^L\!\! {v}\cos\frac{k\pi x}{L}dx
\end{pmatrix}\!\!=\!\!\begin{pmatrix}
0\\
0
\end{pmatrix},
\]
where the eigenvalue $\lambda$ satisfies
\[\bar p_k(\lambda)=\lambda^2+\bar T_k\lambda+\bar D_k=0,\]
with
\[\bar T_k=\big(D_1+D_2\big)\Big(\frac{k\pi}{L}\Big)^2+b_1\bar u+c_2\bar v>0,\]
and
\[\bar D_k=\Big(D_1\big(\frac{k\pi}{L}\big)^2+b_1\bar{u}\Big)\Big(D_2\big(\frac{k\pi}{L}\big)^2 +c_2\bar v\Big)-\Big(\chi_k\bar{u}\phi(\bar{v})\big(\frac{k\pi}{L}\big)^2+c_1\bar u\Big)b_2 \bar{v}.\]
If $\chi_k\neq\chi_{k_0}=\min_{k\in \mathbb N^+} \chi_k$, $\bar D_k<0$ thanks to (\ref{314}), therefore $\bar p_k(\lambda)$ hence (\ref{337}) always have a positive root $\lambda(0)$ for all $k\neq k_0$.  From the standard eigenvalue perturbation theory in \cite{Ka}, (\ref{336}) always has a positive root $\lambda(s)$ for small $s$ if $k\neq k_0$.  This finishes the proof of the instability part.

To study the stability of $(u_{k_0}(s,x),v_{k_0}(s,x),\chi_{k_0}(s))$, we first note that $\bar p_{k_0}(\lambda)$ or (\ref{337}) with $k=k_0$ has two eigenvalues, one being negative and the other being zero.  Hence we need to investigate the asymptotic behavior of the zero eigenvalue as $s\approx,\neq0$.  According to Corollary 1.13 in \cite{CR2}, there exists an interval $I$ with $\chi_{k_0}\in I$ and $C^1$-smooth functions $(\chi,s):I\times (-\delta,\delta) \rightarrow (\mu(\chi),\lambda(s))$ such that $(\mu(\chi_{k_0}), \lambda(0))=(0,0)$; moreover, $\lambda(s)$ is the only eigenvalue in any fixed neighbourhood of the complex plane origin and $\mu(\chi)$ is the only eigenvalue of the following eigenvalue problem around $\chi_{k_0}$
\begin{equation}\label{338}
D_{(u,v)}\mathcal{F}(\bar{u},\bar{v},\chi)(u,v)=\mu(u,v),~(u,v)\in \mathcal{X} \times \mathcal{X},
\end{equation}
or equivalently
\begin{equation}\label{339}
\left\{
\begin{array}{ll}
D_1 u''+\chi \bar{u}\phi(\bar{v})v''-b_1\bar{u}u-c_1\bar{u}v=\mu u,& x\in(0,L),  \\
D_2 v''-b_2\bar{v}u-c_2\bar{v}v=\mu v,& x\in(0,L),  \\
u'(x)=v'(x)=0,&x=0,L;
\end{array}
\right.
\end{equation}
furthermore, the eigenfunction of (\ref{338}) can be represented by $\big(u(\chi,x),v(\chi,x)\big)$, which depends on $\chi$ smoothly and is uniquely determined by $\big(u(\chi_{k_0},x),v(\chi_{k_0},x)\big)=\big( Q_{k_0} \cos \frac{k_0\pi x}{L},\cos \frac{k_0\pi x}{L} \big)$ and $\big(u(\chi,x),v(\chi,x)\big)$ $-\big( Q_{k_0}\cos \frac{k_0\pi x}{L},\cos \frac{k_0\pi x}{L} \big) \in \mathcal{Z}$, with $Q_{k_0}$ and $\mathcal{Z}$ being given by (\ref{315}) and (\ref{318}) respectively.

Differentiating (\ref{339}) with respect to $\chi$ and then taking $\chi=\chi_{k_0}$, we have that
\begin{equation}\label{340}
\left\{\!\!\!
\begin{array}{ll}
D_1\dot{u}''\!\!-\!\bar{u}\phi(\bar{v})\big(\cos\frac{k_0\pi x}{L}\big)''\!\!+\chi_{k_0}\bar{u}\phi(\bar{v})\dot{v}''\!\!-\!b_1\bar{u}\dot{u}\!-\!c_1\bar{u}\dot{v}=\dot{\mu}(\chi_{k_0}) Q_{k_0} \cos\frac{k_0\pi x}{L},  \\
D_2 \dot{v}''\!\!-b_2\bar{v}\dot{u}-c_2\bar{v}\dot{v}=\dot{\mu}(\chi_{k_0}) \cos\frac{k_0\pi x}{L},\\
\dot{u}'(x)=\dot{v}'(x)=0,~x=0,L,
\end{array}
\right.
\end{equation}
where $\dot{}$ in (\ref{340}) denotes the derivative taken with respect to $\chi$ and evaluated at $\chi=\chi_{k_0}$, \emph{i.e.}, $\dot{u}=\frac{\partial u(\chi,x)}{\partial \chi}\big\vert_{\chi=\chi_{k_0}}$, $\dot{v}=\frac{\partial v(\chi,x)}{\partial \chi}\big\vert_{\chi=\chi_{k_0}}$.

Testing (\ref{340}) by $\cos\frac{k \pi x}{L}$ over $(0,L)$ yields
\[\begin{pmatrix}\!\!
-D_1 \!\big(\frac{k \pi}{L}\big)^2\!\!-\!b_1\bar{u} & -\chi_{k_0} \big(\frac{k \pi}{L}\big)^2 \!\bar{u}\phi(\bar{v})\!-\!c_1\bar{u}   \\
-b_2\bar{v} & -D_2\big(\frac{k \pi}{L}\big)^2\!\!-\!c_2\bar{v}
\!\!\end{pmatrix}\!
\!\!\begin{pmatrix}\!
\int_0^L \!\!\dot{u}\cos\frac{k_0\pi x}{L}dx \\
~~\\
\int_0^L\!\! \dot{v}\cos\frac{k_0\pi x}{L}dx
\end{pmatrix}\!\!=\!\!\begin{pmatrix}
\!\Big(\dot{\mu}(\chi_{k_0})Q_{k_0}\!-\!\bar{u}\phi(\bar{v})\big(\frac{k_0\pi }{L} \big)^2\Big)\! \frac{L}{2}\\
~~\\
\dot{\mu}(\chi_{k_0})\frac{L}{2}
\!\!\!\!\end{pmatrix}.
\]
The coefficient matrix is singular thanks to (\ref{314}), then we must have that $\frac{D_1 (\frac{k_0 \pi}{L})^2+b_1\bar{u}}{b_2\bar{v}}=\frac{\dot{\mu}(\chi_{k_0})Q_{k_0}-\bar{u}\phi(\bar{v}) (\frac{k_0\pi }{L})^2}{\dot{\mu}(\chi_{k_0})}$ which, in light of (\ref{315}), implies that
\[\dot{\mu}(\chi_{k_0})=-\frac{b_2\bar{u}\phi(\bar{v})\bar{v}\big(\frac{k_0\pi }{L} \big)^2}{(D_1+D_2)\big(\frac{k_0\pi }{L} \big)^2+b_1\bar{u}+c_2\bar{v}}<0.\]
According to Theorem 1.16 in \cite{CR2}, the functions $\lambda(s)$ and $-s\chi'_{k_0}(s)\dot{\mu}(\chi_{k_0})$ have the same zeros and signs near $s=0$, and
\[\lim_{s\rightarrow 0}\frac{-s\chi'_{k_0}(s)\dot{\mu}(\chi_{k_0})}{\lambda(s)}=1, \text{for} \lambda(s) \neq0,\]
then we conclude that $\text{sgn} (\lambda(0))=\text{sgn}(K_2)$ in light of $K_1=0$.  Following the standard perturbation theory, one can show that $\text{sgn} (\lambda(s))=\text{sgn}(K_2)$ for $s\in(-\delta,\delta)$, $\delta$ being small.  This finishes the proof of Theorem \ref{theorem33}.
\end{proof}

According to Theorem \ref{theorem33}, if $k_0$ is the positive integer that minimizes $\chi_k$ over $\mathbb N^+$, the bifurcation branch $\Gamma_{k_0}(s)$ around $(\bar u,\bar v,\chi_{k_0})$ is stable if it turns to the left and it is unstable if it turns to the right.  However, for all $k\neq k_0$, $\Gamma_k(s)$ is always unstable for $s$ being small.  It is unknown about the global behavior of the continuum of $\Gamma_k(s)$.  We discuss this in details in Section \ref{section5}.

Proposition \ref{proposition1} indicates that $(\bar u,\bar v)$ loses its stability as $\chi$ surpasses $\chi_{k_0}$.  Theorem \ref{theorem33} shows that the stability is lost to the nonconstant steady state $(u_{k_0}(s,x),v_{k_0}(s,x))$ in the form $(Q_{k_0},1)\cos \frac{k_0\pi x}{L}$ and we call this the stable \emph{wavemode} of (\ref{1}).  If the initial value $(u_0,v_0)$ being a small perturbation from $(\bar u,\bar v)$, then spatially inhomogeneous patterns can emerge through this mode, at least when $\chi$ is around $\chi_{k_0}$, therefore stable patterns with interesting structures, such as interior spikes, transition layers, etc. can develop as time evolves.  Rigorous analysis of the boundary spike solutions will be analyzed in the coming section.

If the interval length $L$ is small, we see that
\[\chi_k=\frac{\big( D_1(\frac{k\pi}{L})^2+b_1\bar{u} \big)\big(D_2(\frac{k\pi}{L})^2+c_2\bar{v} \big)-b_2c_1\bar{u}\bar{v} }{b_2(\frac{k\pi}{L})^2\bar{u}\phi(\bar{v})\bar{v}}\approx \frac{D_1D_2(\frac{k\pi}{L})^2}{b_2\phi(\bar v)\bar u\bar v},\]
therefore $\chi_1=\chi_{k_0}=\min_{k\in \mathbb N^+} \chi_k$ and Theorem \ref{theorem33} shows that the only stable pattern is $(u_1(s,x),v_1(s,x))$ which is spatially monotone.  It is easy to see that $k_0$ increases if $L$ increases.  This indicates that small domain only supports monotone stable solutions, while large domain supports non-monotone stable solutions.  Indeed, one can construct non-monotone solutions to (\ref{31}) by reflecting and periodically extending the monotone ones at the boundary points $0,\pm L,\pm 2L,...$


\section{Boundary spikes with limiting diffusion rates}\label{section4}
This section is devoted to investigate positive solutions to (\ref{31}) with large amplitude compared to the small amplitude bifurcating solutions.  In particular, we study the solution profiles as $D_1$ and $\chi$ approach infinity with $\chi/D_1\in (0,\infty)$ being fixed.  Since we will show that small $D_2$ gives rise to boundary spikes for system (\ref{31}), we put $\epsilon=\sqrt D_2$, and we put $\phi(v)\equiv1$ in (\ref{31}) without loss of our generality.  It is the goal of this section to investigate positive solutions with large amplitude to the following system
\begin{equation}\label{41}
\left\{
\begin{array}{ll}
(D_1 u'+\chi u v')'+\big(-1+\frac{1}{a_1+b_1u+c_1v}\big)u=0,&x \in (0,L),\\
\epsilon^2  v''+\big(-1+\frac{1}{a_2+b_2u+c_2v}\big)v=0,&x \in (0,L),\\
u'(x)=v'(x)=0,&x=0,L,
\end{array}
\right.
\end{equation}
Our main results are the followings.
\begin{theorem}\label{theorem41}
Let $r=\frac{\chi}{D_1}\in(0,\infty)$ and $a_2\in(0,1)$ be fixed.  For any $\epsilon>0$ being small, we can find $\bar D>0$ large such that if $D_1>\bar D$, there always exists a nonconstant positive solution $(u,v)$ to (\ref{41}).  Moreover, as $D_1\rightarrow \infty$, $(u(x),v(x))$ converges to $(\lambda_{\epsilon} e^{-rv_{\epsilon}(x)}, v_{\epsilon}(x))$ uniformly in $[0,L]$, where $\lambda_{\epsilon}$ is a positive constant and $\lambda_{\epsilon} \rightarrow 0$ as ${\epsilon} \rightarrow 0$; $v_{\epsilon}(x)$ is a positive function of $x$ and $v_{\epsilon}(x)\rightarrow \frac{1-a_2}{c_2}$ compact uniformly on $(0,L]$ and $v_{\epsilon}(0)\rightarrow \frac{1-a_2}{2c_2}$.
\end{theorem}
According to Theorem \ref{theorem41}, $u$ has a boundary spike and $v$ has a boundary layer at $x=0$ if $D_1$ and $\chi$ are sufficiently large and $D_2$ is sufficiently small.  The elliptic spiky solutions can be used to model segregation phenomenon through inter--specific competitions.

\subsection{Convergence to shadow system}
Theorem \ref{theorem41} is an immediate consequence of several preliminary results.  We first study (\ref{41}) by passing $D_1$ to infinity.  To this end, we need the following a prior estimates.
\begin{lemma}\label{lemma42}
Let $(u,v)$ be any positive solution to (\ref{41}).  Then there exists a positive constant $C_0$ independent of $D_1$ and $\chi$ such that
\begin{equation}\label{42}
0<\max_{x\in[0,L]} v(x), \Vert v(x) \Vert_{C^2([0,L])}\leq C_0;
\end{equation}
moreover, if $\frac{\chi}{D_1}$ is bounded, there exists $C_1>0$ such that
\begin{equation}\label{43}
\Vert u'(x) \Vert_{L^2}\leq C_1.
\end{equation}
\end{lemma}
\begin{proof}
We have from the Maximum Principles that
\[\max_{x\in[0,L]} v(x)\leq\frac{1-a_2}{c_2};\]
therefore $v$ is bounded in $C^2([0,L])$ thanks to the standard elliptic Schauder estimates.

On the other hand, we integrate the $u$-equation over $(0,L)$ to see that
\[\int_0^L u(x)dx=\int_0^L \frac{u(x)}{a_1+b_1u(x)+c_1v(x)}dx\leq \frac{L}{b_1}.\]
Testing the $u$-equation by $u$ and integrating it over $(0,L)$ give rise to
\[
\begin{split}
D_1\int_0^L (u')^2dx&=-\chi\int_0^L u u'v' dx+\int_0^L f(u,v)udx\\
&\leq \frac{D_1}{2} \int_0^L (u')^2dx+\Big(\frac{\chi^2\Vert v'\Vert_{L^\infty}}{2D_1}+\frac{1}{a_1}-1\Big)\int_0^L u^2dx,
\end{split}\]
therefore $\Vert u'(x) \Vert_{L^2}$ is bounded for finite $\frac{\chi}{D_1}$ thanks to the Gagliardo-Nirenberg inequality: $\Vert u \Vert^2_{L^2(0,L)}\leq \delta \Vert u' \Vert^2_{L^2(0,L)}+C(\delta)\Vert u \Vert^2_{L^1(0,L)}$, where $\delta>0$ is an arbitrary constant and $C(\delta)>0$ only depends on $\delta$ and $L$.
\end{proof}

We now study the asymptotic behaviors of positive solutions $(u,v)$ to (\ref{41}) by passing advection rate $D_1$ and advection rate $\chi$ to infinity.  We assume that $\frac{\chi}{D_1}$ remains bounded in this process, therefore both $u$ and $v$ are bounded as in Lemma \ref{lemma42}.
\begin{proposition}\label{proposition2}
Let $(u_i,v_i)$ be positive solutions of (\ref{41}) with $(D_{1,i},\epsilon_i,\chi_i)=(D_1,\epsilon,\chi)$.  Denote $\frac{\chi_i}{D_{1,i}}=r_i$.  Assume that $\chi_{i} \rightarrow \infty$, $\epsilon_i\rightarrow \epsilon \in(0,\infty)$ and $r_i \rightarrow r\in(0,\infty)$ as $i\rightarrow \infty$, then there exists a nonnegative constant $\lambda_\epsilon$ such that $u_ie^{r_iv_i} \rightarrow \lambda_\epsilon$ uniformly on $[0,L]$ as $i \rightarrow \infty$; moreover, $(u_i,v_i)\rightarrow (\lambda_\epsilon  e^{-r v_\epsilon },v_\epsilon)$ in $C^1([0,L])\times C^1([0,L])$ after passing to a subsequence if necessary, where $v_\epsilon=v_\epsilon(x)$ satisfies the following shadow system
\begin{equation}\label{44}
\left\{
\begin{array}{ll}
\epsilon^2 v''_\epsilon+(-1+\frac{1}{a_2+b_2\lambda_\epsilon e^{-r v_\epsilon }+c_2v_\epsilon})v_\epsilon =0, & x \in(0,L),\\
\int_0^L (-1+\frac{1}{a_1+b_1\lambda_\epsilon e^{-r v_\epsilon }+c_1v_\epsilon})\lambda_\epsilon e^{-rv_\epsilon}dx=0,\\
v_\epsilon'(0)=v_\epsilon'(L)=0.
\end{array}
\right.
\end{equation}
\end{proposition}
\begin{proof} Since $v_i$ is bounded in $C^2([0,L])$ uniformly for all $\chi_i$ and $D_{1,i}>0$.  By the compact embedding, $v_i $ converges to some $v_\infty$ in $C^1([0,L])$ as $i\rightarrow \infty$, after passing to a subsequence if necessary.  On the other hand, we integrate the $u$-equation in (\ref{41}) over $(0,x)$ to have
\[u_i'+r_i u_iv'_i= \frac{1}{D_{1,i}} \int_0^x (1-\frac{1}{a_1+b_1u_i+c_1v_i})u_idx.\]
Denoting $w_i=u_i e^{r_iv_i}$, we have that
\begin{equation}\label{45}
\Big \vert e^{-r_iv_i}w_i' \Big \vert\leq \frac{1+\frac{1}{a_1}}{D_{1,i}} \int_0^x u_i dx.
\end{equation}
Sending $i$ to $\infty$ in (\ref{45}), we conclude that $w'_{i} \rightarrow 0$ uniformly on $[0,L]$.  Therefore, $w_{i}=u_ie^{r v_i}$ converges to a nonnegative constant $\lambda_\epsilon$.  Moreover we can use standard elliptic regularity theory to show that $v_\epsilon$ is $C^\infty$-smooth and it satisfies the shadow system (\ref{44}).
\end{proof}
 Proposition \ref{proposition2} implies that when both $D_1$ and $\chi$ are sufficiently large, the steady state $(u(x),v(x))$ can be approximated by the structures of the shadow system (\ref{44}).  We want to remark that the arguments in this Proposition carry over to the case of higher-dimensional bounded domains.  To find boundary spikes of (\ref{1})--(\ref{2}), we study the asymptotic behavior of the shadow system with small diffusion rate $\epsilon$ and then investigate the original system with large diffusion and advection rates.  This approach is due to the idea from \cite{Ke}, applied by \cite{NT,T} in reaction--diffusion systems and developed by \cite{LN,LN2} for reaction--diffusion system with cross--diffusions.

\subsection{Boundary spike layers of the shadow system}
We proceed to construct boundary--spike solutions to (\ref{44}).  Our results state as follows.
\begin{proposition}\label{proposition3}
Assume that $r\in(0,\infty)$ and $a_2\in(0,1)$.  Then there exists a $\epsilon_0 >0$ such that (\ref{44}) has a nonconstant positive solution $(\lambda_\epsilon,v_\epsilon(\lambda_\epsilon,x))$ for all $\epsilon \in(0,\epsilon_0)$; moreover, $\lambda_\epsilon \rightarrow 0$ as $\epsilon\rightarrow 0$ and
\begin{equation}\label{48}
 v_\epsilon(\lambda_\epsilon,x)\rightarrow
\left\{
\begin{array}{ll}
v^*=\frac{1-a_2}{c_2},&\text{compact uniformly for~} x\in(0,L],\\
\frac{1-a_2}{2c_2}, &x=0.
\end{array}
\right.
\end{equation}
\end{proposition}
For $\epsilon$ being sufficiently small, $v_\epsilon(\lambda_\epsilon;x)$ has a single inverted boundary spike at $x=0$, then one can construct solutions to (\ref{44}) with multi- boundary and interior spikes by periodically reflecting and extending $v_\epsilon(x)$ at $x=\pm L$, $\pm2L$,...

To prove Proposition \ref{proposition3}, we first choose $\lambda>0$ to be a predetermined fixed constant and establish nonconstant positive solutions $v_\epsilon(\lambda;x)$ to the following problem
\begin{equation}\label{46}
\left\{
\begin{array}{ll}
\epsilon^2 v''_\epsilon+(-1+\frac{1}{a_2+b_2\lambda  e^{-rv_\epsilon}+c_2v_\epsilon})v_\epsilon =0, & x \in(0,L),\\
v_\epsilon(x)>0, x\in(0,L); v_\epsilon'(0)=v_\epsilon'(L)=0
\end{array}
\right.
\end{equation}
then we find $\lambda=\lambda_\epsilon>0$ such that $(\lambda_\epsilon,v_\epsilon(\lambda_\epsilon;x))$ satisfies the integral constraint
\begin{equation}\label{47}
\int_0^L \Big(-1+\frac{1}{a_1+b_1\lambda_\epsilon e^{-rv_\epsilon}+c_1v_\epsilon}\Big)\lambda_\epsilon e^{-rv_\epsilon} dx=0.
\end{equation}

Denote
\begin{equation}\label{49}
w_\epsilon(x)=v^*-v_\epsilon(x)=\frac{1-a_2}{c_2}-v_\epsilon(x)
\end{equation}
then (\ref{46}) becomes
\begin{equation}\label{410}
\left\{
\begin{array}{ll}
\epsilon^2 w''_\epsilon+ f(\lambda;w_\epsilon)=0, & x \in(0,L),\\
w_\epsilon'(0)=w_\epsilon'(L)=0,
\end{array}
\right.
\end{equation}
where
\begin{equation}\label{411}
f(\lambda;w_\epsilon)=-\Big(-1+\frac{1}{1+b_2\lambda e^{-r(v^*-w_\epsilon)}-c_2w_\epsilon}\Big)\Big(v^*-w_\epsilon\Big).
\end{equation}
It is equivalent to study the existence and asymptotic behaviors of (\ref{410}) in order to prove Proposition \ref{proposition3}.  We first collect some facts about $f(\lambda;s)$ introduced in (\ref{411}).  We denote
\[{f}(\lambda;s)=\frac{g(\lambda;s)(v^*-s)}{1+g(\lambda;s)}, s\in (0,\infty),\]
where
\begin{equation}\label{412}
g(\lambda;s)=b_2\lambda e^{-r(v^*-s)}-c_2s.
\end{equation}
\begin{lemma}\label{lemma43}
Let $r\in(0,\infty)$.  For each $\lambda\in(0,\frac{c_2}{b_2r}e^{rv^*-\frac{r}{c_2}-1})$, $g(\lambda;s)$ has two positive roots $s_1<s_2$ and $1+g(\lambda;s)$ has two positive roots $s_3<s_4$ such that $s_1<s_3<s_4<s_2$; moreover as $\lambda \rightarrow 0$, $(s_1,s_3)\rightarrow (0,\frac{1}{c_2})$ and $s_2$, $s_4\rightarrow \infty$.
\end{lemma}
\begin{proof}
It is easy to see that the roots of $g(s)$ or $1+g(s)$ must be positive if they exist; moreover, if $1+g(s)$ has positive roots, so does $g(s)$.  Hence we only need to show that $1+g(s)$ has positive roots and it is equivalent to show that its minimum over $\mathbb R^+$ is negative.  Let $s^*$ be the critical point of $1+g(s)$, we have from straightforward calculations that
\[
\left\{
\begin{array}{ll}
g(0)=b_2\lambda e^{-rv^*}>0,\\
g'(s^*)=0,~s^*=\frac{1}{r}\ln(\frac{c_2e^{rv^*}}{b_2\lambda r})>\frac{1}{c_2}+\frac{1}{r},\\
g^{''}(s)>0,~s\in (-\infty,\infty),
\end{array}
\right.
\]
therefore $1+g(s)$ is convex and its minimum value is $1+g(s^*)=1+\frac{c_2}{r}-\frac{c_2}{r}\ln(\frac{c_2e^{rv^*}}{b_2\lambda r})<0$ as desired.  Since $s_i$ is continuous in $\lambda$, putting $\lambda=0$ in $1+g(s)$ gives $s_1=0$ and $s_3=\frac{1}{c_2}$.  Moreover, $s_2>s_4>s^* \rightarrow \infty$ as $\lambda\rightarrow 0$.  This finishes the proof.
\end{proof}
Graphes of $g$ and $1+g$ are presented in Figure \ref{figure1}.
\begin{figure}[h!]
\centering
\includegraphics[width=2.4in,height=2.4in]{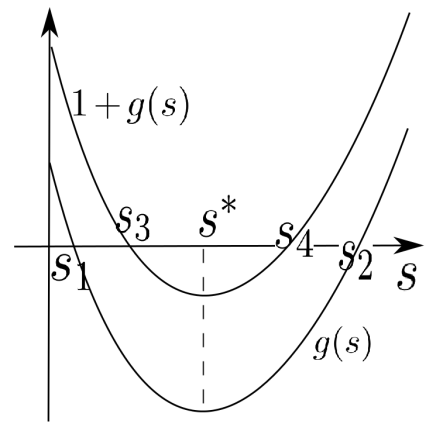}
\caption{Graphes of $g(s)$ and $1+g(s)$.}\label{figure1}
\end{figure}
\begin{lemma}\label{lemma44}
Let $r\in(0,\infty)$.  Then $s_1<v^*<s_3$ if
\begin{equation}\label{413}
\lambda<\min\Big\{\frac{c_2}{b_2r}e^{rv^*-\frac{r}{c_2}-1},\frac{1-a_2}{b_2}\Big\}.
\end{equation}
\end{lemma}
\begin{proof}
We only need to show that $-1=g(s_3)<g(v^*)<g(s_1)=0$ and $g(s^*)<-1$, where we already have from above that $g(v^*)=b_2\lambda-c_2v^*$; moreover, $g(s^*)=\frac{c_2}{r}-\frac{c_2}{r}\ln(\frac{c_2e^{rv^*}}{b_2\lambda r})<-1$.
\end{proof}
We shall assume condition (\ref{413}) from now on.  Let us introduce the notations
\[\tilde w(x)=w(x)-s_1\]
and
\begin{equation}\label{414}
\tilde f(\lambda;s)=f(\lambda;s+s_1)=-\Big(-1+\frac{1}{1+b_2\lambda e^{-r(v^*-s_1-s)}-c_2(s+s_1)}\Big)\Big(v^*-s_1-s\Big),
\end{equation}
then $\tilde w$ satisfies
\begin{equation}\label{415}
\left\{
\begin{array}{ll}
\epsilon^2 \tilde w''+\tilde f(\lambda;\tilde w)=0,&x\in(0,L),\\
\tilde w'_\epsilon(0)=\tilde w'_\epsilon(L)=0.
\end{array}
\right.
\end{equation}
Moreover, if $\tilde w_\epsilon(x)$ is a positive solution to (\ref{415}) then
\begin{equation}
v_\epsilon(x)=v^*-s_1-\tilde w_\epsilon
\end{equation}
is a solution to (\ref{46}).
\begin{lemma}\label{lemma45}
For each $\lambda$ satisfying (\ref{413}), the following problem has a unique solution $W_0=W_0(\lambda;z)$
\begin{equation}\label{417}
\left\{
\begin{array}{ll}
W_0 ''+\tilde f(\lambda; W_0)=0,z\in \mathbb R,\\
W_0(\lambda;0)=\frac{v^*-s_1}{2}>0, z\in (0,\infty); W'_0(\lambda;0)=0,\\
\lim_{z\rightarrow \pm\infty} W_0(\lambda;z)=0,
\end{array}
\right.
\end{equation}
such that $W_0\in C^2(\mathbb R)$, $W_0(\lambda;z)>0$ in $\mathbb R$ and $W_0'(\lambda;z)z<0$ for $z\in\mathbb R\backslash \{0\}$.  Moreover, $W_0(\lambda;z)$ is radially symmetric and $W_0$, $W'_0$, $W''_0$ decays exponentially at $\infty$ uniformly in $\lambda$, \emph{i.e.}, there exists $C_0$, $\eta>0$ independent of $\lambda$ such that
\[0\leq W_0(\lambda;z), \vert W'_0(\lambda;z)\vert, \vert W''_0(\lambda;z)\vert \leq C_0 e^{-\eta \vert z\vert}, z\in \mathbb R. \]
\end{lemma}

\begin{proof}
We introduce the transformation
\[\tilde F(\lambda; s)=\int_0^s \tilde f(\lambda;t)dt=\int_{s_1}^{s+s_1} f(\lambda;t)dt.\]
therefore $\tilde F(\lambda; s)<0$ for $s\in(0,v^*-s_1)$--see the graph of $\tilde f$ for example; on the other hand, we can show that $\lim_{s\rightarrow (s_3-s_1)^{-}}\tilde F(\lambda; s)=+\infty$ and according to Intermediate value theorem, there exists $s_0\in(v^*-s_1,s_3-s_1)$ such that $\tilde F(\lambda; s_0)=0$; moreover, $s_0=\inf \{s>0~\vert~\tilde F(\lambda; s)=0\}$ and $\tilde f(\lambda; s_0)>0$.  Therefore $\tilde f$ satisfies condition (6.2) in \cite{BL} and Theorem 5 there implies our existence results.  Moreover, the exponential decay follows from Remark 6.3 in \cite{BL}, since $\tilde f'(\lambda; 0)<0$ for all $\lambda$ satisfying (\ref{413}).
\end{proof}

By the unique solution $W_0$ to (\ref{417}), we now construct a boundary spike to (\ref{414}) hence a boundary layer to (\ref{46}).  To this end, we choose a smooth cut--off function $\rho (x)$ such that $\rho (x)\equiv1$ for $\vert x \vert \leq \frac{L}{3}$, $\rho (x) \equiv 0$ for $\vert x \vert \geq \frac{2L}{3}$ and $\rho(x)\in[0,1]$ for $x\in \mathbb R$.  Denoting
\[W_{\epsilon,\lambda}(x)=\rho (x)W_0(\lambda; x/\epsilon),\]
we want to prove that (\ref{415}) has a solution in the form $\tilde w_\epsilon(\lambda,x)=W_{\epsilon,\lambda}(x)+\epsilon \psi(\lambda,x)$. Then $\psi$ satisfies
\begin{equation}\label{418}
\mathcal{L}_\epsilon \psi+\mathcal{P}_\epsilon+\mathcal{Q}_\epsilon=0,
\end{equation}
where
\begin{equation}\label{419}
\mathcal{L}_\epsilon=\epsilon^2\frac{d^2}{dx^2}+\tilde f_w(\lambda; W_{\epsilon,\lambda}(x)),
\end{equation}
\begin{equation}\label{420}
\mathcal{P}_\epsilon=\epsilon^{-1}\Big(\epsilon^2\frac{d^2}{dx^2}W_{\epsilon,\lambda}+\tilde{f}(\lambda; W_{\epsilon,\lambda})\Big),
\end{equation}
and
\begin{equation}\label{421}
\mathcal{Q}_\epsilon=\epsilon^{-1}\Big(\tilde{f}(\lambda; W_{\epsilon,\lambda}+\epsilon\psi)-\tilde{f}(\lambda; W_{\epsilon,\lambda})-\epsilon\psi\tilde{f}_w(\lambda; W_{\epsilon,\lambda})\Big).
\end{equation}
According to (\ref{418})--(\ref{421}), $\mathcal{P}_\epsilon$ and $\mathcal{Q}_\epsilon$ measure the accuracy that $W_{\epsilon,\lambda}(x)$ approximates solution $w_\epsilon(\lambda,x)$.  Our existence result is a consequence of several lemmas.
Set
\[C^2_\textbf{n}([0,L])=\{u\in C^2([0,L]): u'(0)=u'(L)=0\},\]
then we first present the following set of results.
\begin{lemma}\label{lemma46}
Let $r\in(0,\infty)$.  Suppose that $a_2\in(0,1)$ and $\lambda$ satisfies (\ref{413}).  Then there exists a small $\epsilon_0$ such that if $\epsilon\in(0,\epsilon_0)$, $\mathcal L_\epsilon$ with domain $C^2_\textbf{n}([0,L])$ has an inverse $\mathcal L_{\epsilon}^{-1}$; moreover,  $\mathcal L_{\epsilon}^{-1}: C([0,L])\rightarrow C([0,L])$ is uniformly bounded in $\epsilon$, \emph{i.e.}, there exists $C_0$ independent of $\epsilon$ such that
\begin{equation}\label{422}
\sup_{x\in[0,L]}\vert \mathcal L_{\epsilon}^{-1} f \vert \leq C_0 \sup_{x\in[0,L]}\vert f \vert, \forall f\in C([0,L]).
\end{equation}
\end{lemma}
\begin{lemma}\label{lemma47}
Suppose that the conditions in Lemma \ref{lemma46} hold.  Then there exist $C_1>0$ and small $\epsilon_1>0$ such that for all $\epsilon\in(0,\epsilon_1)$
\[\sup_{x\in[0,L]} \vert\mathcal{P}_\epsilon (x)\vert \leq  C_1.\]
\end{lemma}
\begin{lemma}\label{lemma48}
Suppose that the conditions in Lemma \ref{lemma46} hold.  For each $R_0>0$, denote $\mathcal B(R_0)=\{f\vert \sup_{x\in[0,L]} \vert f \vert<R_0\}$.  Then there exist $C_2=C_2(R_0)>0$ and small $\epsilon_2=\epsilon_2(R_0)>0$ such that for all $\epsilon\in (0,\epsilon_2)$,
\begin{equation}\label{423}
\sup_{x\in[0,L]} \vert\mathcal{Q}_\epsilon [\psi_i]\vert \leq C_2 \epsilon \sup_{x\in[0,L]}\vert \psi_i \vert,~ \forall \psi_i \in \mathcal B(R_0),
\end{equation}
\begin{equation}\label{424}
\sup_{x\in[0,L]} \vert\mathcal{Q}_\epsilon [\psi_1]-\mathcal{Q}_\epsilon [\psi_2]\vert \leq C_2 \epsilon \sup_{x\in[0,L]} \vert \psi_1-\psi_2\vert,~ \forall \psi_1, \psi_2 \in \mathcal B(R_0).
\end{equation}
\end{lemma}
We want to point out that Lemma \ref{lemma46} generalizes Lemma 5.3 in \cite{WGY} which holds for $L^p$, $p>1$.  Assuming Lemmas \ref{lemma46}--\ref{lemma48}, we prove the following results of positive solutions to (\ref{415}).
\begin{proposition}\label{proposition4}
Let $r\in(0,\infty)$ and $a_2\in(0,1)$.  Suppose that (\ref{413}) is satisfied.  There exists a small $\epsilon_4>0$ such that for all $\epsilon\in (0,\epsilon_4)$, (\ref{415}) has a positive solution $\tilde w_\epsilon(\lambda;x)\in C([0,L])$ such that
\[\sup_{x\in(0,L)} \vert \tilde w_\epsilon(\lambda;x)-W_\epsilon(\lambda;x) \vert \leq C_4  \epsilon,\] where $C_4$ is a positive constant independent of $\epsilon$.  In particular,
\begin{equation}\label{425}
\lim_{\epsilon \rightarrow 0^+} \tilde w_\epsilon(\lambda;x)=
\left\{
\begin{array}{ll}
0, \text{ compact uniformly on } (0,L],    \\
v^*-s_1=\frac{1-a_2}{c_2}-s_1>0,~x=0,
\end{array}
\right.
\end{equation}
where $s_1$ is the positive root of $g(\lambda;s)=0$ obtained in Lemma \ref{lemma43}.
\end{proposition}
\begin{proof}
We shall apply the Fixed point theorem on (\ref{418}) to show that $\tilde w_\epsilon \in C([0,L])$ takes the form of $\tilde w_\epsilon=W_{\epsilon,\lambda}+\epsilon \psi$ for a smooth function $\psi$.  We define
\begin{equation}\label{426}
\mathcal{S}_\epsilon[\psi]=-\mathcal{L}^{-1}_\epsilon(\mathcal{P}_\epsilon+\mathcal{Q}[\psi]).
\end{equation}
Then $\mathcal{S}_\epsilon$ is a bounded linear operator from $C([0,L])$ to $C([0,L])$ uniform in $\lambda$.  Moreover, choosing $R_0\geq 2C_0C_1$, we have that $\sup_{x\in(0,L)} \vert \mathcal{L}^{-1} \mathcal{P}_\epsilon \vert \leq C_0C_1$ thanks to Lemma \ref{lemma46} and \ref{lemma47}.  Therefore, it follows from Lemma \ref{lemma48} that if $\epsilon$ is small,
\[\sup_{x\in(0,L)} \vert \mathcal{S}_\epsilon[\psi]\vert \leq C_0C_1+C_0C_2\epsilon R_0\leq R_0,~\forall \psi\in \mathcal{B}(R_0).\]
and
\[\sup_{x\in(0,L)} \vert \mathcal{S}_\epsilon[\psi_1]-\mathcal{S}_\epsilon[\psi_2]\vert \leq \frac{1}{2} \sup_{x\in(0,L)} \vert \psi_1-\psi_2 \vert,~\forall \psi_1,\psi_2 \in \mathcal{B},\]
hence $\mathcal{S}_\epsilon$ is a contraction mapping on $\mathcal{B}(R_0)$ for small positive $\epsilon$.  We conclude from the Banach Fixed Point Theorem that $\mathcal{S}_\epsilon$ has a fixed point $\psi_\epsilon$ in $\mathcal{B}$, which is a smooth solution of (\ref{415}).  It is easy to show that $\tilde w_\epsilon$ satisfies (\ref{425}) and this finishes the proof of Proposition \ref{proposition4}.
\end{proof}
Thanks to (\ref{413})--(\ref{415}), $v_\epsilon(x)=v^*-s_1-\tilde w_\epsilon(x)$ is a smooth solution of (\ref{411}).  We have the following results.
\begin{corollary}\label{corollary1}
Under the same condition as in Proposition \ref{proposition4}.  There exists $\epsilon_4>0$ such that for all $\epsilon\in (0,\epsilon_4(\delta))$, (\ref{46}) has a positive solution $\tilde v_\epsilon(\lambda;x)\in C([0,L])$ satisfying
\begin{equation}\label{427}
\lim_{\epsilon \rightarrow 0^+} \tilde v_\epsilon(\lambda;x)=
\left\{
\begin{array}{ll}
v^*-s_1=\frac{1-a_2}{c_2}-s_1, \text{ compact uniformly on } (0,L],    \\
(v^*-s_1)/2,~x=0,
\end{array}
\right.
\end{equation}
\end{corollary}

Now we prove Lemmas \ref{lemma46}-\ref{lemma48}.
\begin{proof}[Proof\nopunct]\emph{of Lemma} \ref{lemma46}.
We argue by contradiction.  Choose a positive sequence $\{\epsilon_i\}_{i=1}^\infty$ with $\epsilon_i \searrow 0$ as $i\rightarrow \infty$.  Suppose that there exist $\Psi_i(x) \in C^2([0,L])$ and $h_i(x)\in C([0,L])$ satisfying
\begin{equation}\label{428}
\left\{
\begin{array}{ll}
\mathcal L_{\epsilon_i} \Psi_i=h_i, & x \in(0,L),\\
\Psi_i'(0)=\Psi'_i(L)=0,
\end{array}
\right.
\end{equation}
such that $\sup_{x\in[0,L]} \vert\Psi_i(x)\vert=1$ and $\sup_{x\in[0,L]} \vert h(x) \vert\rightarrow 0$ as $i\rightarrow \infty$.  Define
\[\tilde \Psi_i(z)=\Psi_i(\epsilon x),\tilde h_i(z)=h_i(\epsilon_ix), \text{ and } \tilde W_{\epsilon_i,\lambda}(z)=W_{\epsilon_i,\lambda}(\epsilon_i x),\]
then
\[\frac{d^2 \tilde \Psi_i }{dz^2}+\tilde f_w (\lambda;\tilde W_{\epsilon_i,\lambda}(z)) \tilde\Psi_i=\tilde h_i, z\in\Big(0,\frac{L}{\epsilon_i}\Big).\]
Let $R_0<\frac{L}{\epsilon_i}$ be an arbitrarily chosen but fixed constant.  Without loss of our generality, we assume $\sup_{[0,L]} \Psi_i(x)=1$.  Then we infer from the boundedness of $\{\tilde h_i(z)\}_{i=1}^\infty$ and the elliptic Schauder estimate that $\{\tilde \Psi_i(z)\}_{i=1}^\infty$ is bounded in $C^2([0,R_0])$, therefore there exists a subsequence as $i\rightarrow \infty$ such that $\tilde \Psi_i(z) \rightarrow \tilde \Psi_0(z)$ in $C^{1,\theta}([0,R_0])$ thanks to the compact embedding $C^2([0,R_0]) \subset C^{1,\theta}([0,R_0])$, $\theta \in(0,1)$.  On the other hand, since $\tilde f_w(\lambda; \tilde W_{\epsilon_i,\lambda}(z))\rightarrow \tilde f_w(\lambda;W_0(\lambda;z))$ and $\tilde h_i(z) \rightarrow 0$ uniformly on $[0,R_0]$, $\frac{d^2 \tilde \Psi_i}{dz^2}$ converges in $C([0,R_0])$ hence $\tilde \Psi_i\rightarrow \tilde \Psi_0$ in $C^2([0,R_0])$.  Applying standard diagonal argument and elliptic regularity theory, we can show that $\tilde \Psi_0$ is in $C^\infty(0,\infty)$ and it satisfies
\begin{equation}\label{429}
\left\{
\begin{array}{ll}
\frac{d^2 \tilde \Psi_0}{dz^2}+\tilde f_w (\lambda;W_0(\lambda;z))\tilde \Psi_0=0, & z \in(0,\infty),\\
\tilde \Psi_0'(0)=0,
\end{array}
\right.
\end{equation}
where $W_0$ is the unique solution to (\ref{417}).

Now we show that $\tilde \Psi_0(0)=1$.  To this end, let $Q_i \in[0,L]$ be such that $\tilde \Psi_i(Q_i)=1$.  We claim that $Q_i\leq C_0 \epsilon$ for some bounded $C_0>0$ independent of $\epsilon$.  In order to prove this claim, we argue by contradiction and assume that $\frac{Q_i}{\epsilon} \rightarrow \infty$ as $i\rightarrow \infty$.  Define
\[\hat \Psi_i(z)=\Psi_i(\epsilon z+Q_i) \text{~and~}\hat h_i(z)=h_i(\epsilon_iz+Q_i)\]
for $\epsilon_i z+Q_i\in[0,L]$ or $z\in[-\frac{Q_i}{\epsilon_i}, \frac{L-Q_i}{\epsilon_i}]$,  therefore
\begin{equation}\label{430}
\left\{
\begin{array}{ll}
\frac{d^2 \hat \Psi_i }{dz^2}+\tilde f_w (\lambda;\hat W_{\epsilon_i,\lambda}(z)) \hat \Psi_i=\hat h_i,& z\in[-\frac{Q_i}{\epsilon_i}, \frac{L-Q_i}{\epsilon_i}],\\
\hat \Psi_i(0)=\sup_{z\in[-\frac{Q_i}{\epsilon_i}, \frac{L-Q_i}{\epsilon_i}]} \hat \Psi_i(z)=1, \frac{d\hat \Psi_i(0)}{dz}=0.
\end{array}
\right.
\end{equation}
By the same arguments as above, we can show that $\hat \Psi_i$ converges to some $\hat \Psi_0$ (at least) in $C^2((-\infty,0])$ and $\hat \Psi_0$ is in $C^\infty((-\infty,0])$; moreover, since $\frac{Q_i}{\epsilon_i}\rightarrow \infty$, we have from the exponential decaying property of $W_0$ that $\hat W_{\epsilon_i,\lambda}(z)=\rho(\epsilon_iz+Q_i)W_0(\lambda;\frac{\epsilon_iz+Q_i}{\epsilon}) \rightarrow 0$ and
\[\frac{d^2 \hat \Psi_0 }{dz^2}+\tilde f_w (\lambda;0) \hat \Psi_0=0, \frac{d\hat \Psi_0(0)}{dz}=0, \sup_{z\in(-\infty,0)}\hat \Psi_0(0)=1,\]
then we have from Maximum Principle that $\hat \Psi''_0(0)=-\tilde f_w (\lambda;0) \hat \Psi_0(0)=-\tilde f_w (\lambda;0)\leq0$, however, this is a contradiction to the fact that $\tilde f_w(\lambda;0)<0$.  This proves our claim and we must have that $\tilde \Psi_0(0)=\sup_{z\in[0,\infty)} \tilde \Psi_0(z)=1$ in (\ref{429}).

Differentiate (\ref{417}) with respect to $z$ and we have that
\begin{equation}\label{431}
\left\{
\begin{array}{ll}
\frac{d^3 W_0 }{dz^3}+\tilde f_w (\lambda;  W_0 )\frac{d W_0 }{dz} =0, z\in(-\infty, \infty),\\
\frac{dW_0(0)}{dz}=0, W_0(0)=\sup_{z\in(-\infty,\infty)}W_0(z)>0,
\end{array}
\right.
\end{equation}
Multiplying (\ref{429}) by $\frac{d W_0 }{dz}$ and integrating it over $(0,\infty)$ lead us to
\begin{equation}\label{432}
\int_0^\infty \frac{d^2 \tilde\Psi_0}{dz^2} \frac{d W_0 }{dz}+\tilde f_w (\lambda;  W_0 ) \tilde \Psi_0 \frac{d W_0 }{dz}=0.
\end{equation}
Multiplying (\ref{433}) by $\tilde \Psi_0$ and integrating it over $(0,\infty)$ lead us to
\begin{equation}\label{433}
\int_0^\infty \frac{d^3 W_0 }{dz^3}\tilde\Psi_0+\tilde f_w (\lambda;  W_0) \frac{d W_0 }{dz} \tilde \Psi_0 =0.
\end{equation}
We infer from (\ref{432}) and (\ref{433}) and the integrations by parts that
\[
\begin{split}
0 &= \int_0^\infty  \tilde\Psi_0''W_0'-\tilde\Psi_0W_0'''\\
&= \int_0^\infty (\tilde\Psi_0'W_0')'-\tilde\Psi_0'W_0''-\Big((\tilde\Psi_0W_0'')'-\tilde\Psi_0'W_0''\Big)\\
&= \int_0^\infty (\tilde\Psi_0'W_0')'-(\tilde\Psi_0W_0'')'\\
&=\tilde\Psi_0'W_0'\Big\vert_0^\infty-\tilde\Psi_0W_0''\Big\vert_0^\infty,
\end{split}
\]
where $'$ denotes the derivative against $z$.  In light of the exponential decay of $W'_0$, $W''_0$ at infinity and the fact $W''_0(0)\neq0$, we have that $\tilde \Psi_0(0)=0$,  and this is a contradiction.  The proof of this lemma is finished.
\end{proof}

\begin{proof}[Proof\nopunct]\emph{of Lemma} \ref{lemma47}.
Substituting $W_{\epsilon,\lambda}(x)=\rho(x)W_0(\lambda;\frac{x}{\epsilon})$ into (\ref{420}) gives rise to
\begin{equation}\label{434}
\begin{split}
\mathcal P_\epsilon=&\epsilon^{-1}\Big(\epsilon^2\rho''(x)W_0(\lambda;x/\epsilon)+2\epsilon \rho'(x)(W_0)_x(\lambda;x/\epsilon)\\
&+\rho(x)(W_0)_{xx}(\lambda;x/\epsilon)+\tilde f(\lambda;\rho(x)W_0(\lambda;x/\epsilon))\Big) \\
=&\epsilon^{-1}\Big(\epsilon^2\rho''(x)W_0(\lambda;x/\epsilon)+2\epsilon \rho'(x)(W_0)_x(\lambda;x/\epsilon)\\
&-\rho(x)\tilde f(\lambda;W_0(\lambda;x/\epsilon))+\tilde f(\lambda;\rho(x)W_0(\lambda;x/\epsilon))\Big).
\end{split}
\end{equation}
If $x\in[0,\frac{L}{3}]\cup[\frac{2L}{3},L]$, $\rho(x)\equiv1$ hence $\mathcal P_\epsilon$ is bounded for all $\lambda$.  If $x\in (\frac{L}{3},\frac{2L}{3})$, since $W_0$ decays exponentially at $\infty$, we can also see that $\mathcal P_\epsilon$ is bounded.  Therefore $\sup_{x\in[0,L]}\vert \mathcal P_\epsilon\vert$ is bounded.
\end{proof}

\begin{proof}[Proof\nopunct]\emph{of Lemma} \ref{lemma48}.
We have from the Intermediate Value Theorem that
\begin{equation}\label{435}
\begin{split}
\sup_{x\in[0,L]} \vert \mathcal{Q}_{\epsilon}[\psi](x)\vert &=\epsilon^{-1}\sup_{x\in[0,L]} \Big[\tilde{f}(\lambda,W_{\epsilon,\lambda}+\epsilon\psi)-\tilde{f}(\lambda, W_{\epsilon,\lambda})-\epsilon\psi\tilde{f}_w(\lambda,W_{\epsilon,\lambda})\Big]\\
&=\epsilon^{-1}\sup_{x\in[0,L]} \Big\vert\int_0^1\Big[\tilde{f}_w(\lambda,W_{\epsilon,\lambda}+t\epsilon\psi)-\tilde{f}_w(\lambda,W_{\epsilon,\lambda})\Big](\epsilon\vert\psi\vert)dt\Big\vert\\
&=\epsilon\sup_{x\in[0,L]} \vert\psi^2\vert\int_0^1\int_0^t \sup_{x\in[0,L]}  \Big\vert\tilde{f}_{ww}(\lambda,W_{\epsilon,\lambda}+s\epsilon\psi)\Big\vert dsdt \\
&\leq C\epsilon\sup_{x\in[0,L]}  \vert\psi\vert^2,
\end{split}
\end{equation}
and
\begin{equation}\label{436}
\begin{split}
&\sup_{x\in[0,L]} \vert \mathcal Q_{\epsilon}[\psi_1]-\mathcal Q_{\epsilon}[\psi_2]\vert\\
=&\epsilon\sup_{x\in[0,L]} \Big \vert\int_0^1\int_0^t\tilde{f}_{ww}(\lambda, W_{\epsilon,\lambda}+s\epsilon\psi_1)\psi_1^2-\tilde{f}_{ww}(\lambda, W_{\epsilon,\lambda}+s\epsilon\psi_2)\psi_2^2dsdt\Big \vert\\
=&\epsilon\sup_{x\in[0,L]} \Big\vert\int_0^1\int_0^t\tilde{f}_{ww}(\lambda, W_{\epsilon,\lambda}+s\epsilon\psi_1)\psi_1^2-\int_0^1\int_0^t\tilde{f}_{ww}(\lambda, W_{\epsilon,\lambda}+s\epsilon\psi_2)\psi_1^2\\
&+\int_0^1\int_0^t\tilde{f}_{ww}(\lambda, W_{\epsilon,\lambda}+s\epsilon\psi_2)\psi_1^2-\int_0^1\int_0^t\tilde{f}_{ww}(\lambda, W_{\epsilon,\lambda}+s\epsilon\psi_2)\psi_2^2dsdt\Big \vert\\
=&\epsilon\sup_{x\in[0,L]} \Big\vert\int_0^1\int_0^t\tilde{f}_{ww}(\lambda, W_{\epsilon,\lambda}+s\epsilon\psi_1)-\tilde{f}_{ww}(\lambda, W_{\epsilon,\lambda}+s\epsilon\psi_2)dsdt\Big\vert\vert\psi_1\vert^2\\
&+\epsilon\sup_{x\in[0,L]} \Big \vert\int_0^1\int_0^t\tilde{f}_{ww}(\lambda, W_{\epsilon,\lambda}+s\epsilon\psi_2)dsdt\Big\vert\big\vert\psi_1^2-\psi_2^2\big \vert\\
\leq & \epsilon C_0\sup_{x\in[0,L]} \vert \psi_1\vert^2 \sup_{x\in[0,L]} \vert\psi_1-\psi_2\vert+\epsilon C_1\sup_{x\in[0,L]} \vert\psi_1+\psi_2\vert\sup_{x\in[0,L]} \vert\psi_1-\psi_2\vert\\
\leq &(C_0R^2+2C_1R)\epsilon\sup_{x\in[0,L]} \vert\psi_1-\psi_2\vert\\
\leq &C_2\epsilon \sup_{x\in[0,L]} \vert\psi_1-\psi_2\vert,
\end{split}
\end{equation}
where we have used the fact that $\vert \tilde f_{ww}(\Psi_i)\vert $, $\vert \tilde f_{www}(\Psi_i)\vert $ are bounded for $\Psi_i\in \mathcal B(R_0)$.
\end{proof}

We proceed to prove Proposition \ref{proposition3}.  Therefore we need to find $\lambda=\lambda_\epsilon$ satisfying (\ref{47}) and (\ref{413}).  In particular, we shall show that $\lambda_\epsilon\rightarrow 0$ as $\epsilon \rightarrow 0$.  We define
\begin{equation}\label{437}
G(\epsilon,\lambda)=\int_0^L \Big(-1+\frac{1}{a_1+b_1\lambda e^{-rv_\epsilon}+c_1v_\epsilon}\Big)\lambda e^{-rv_\epsilon} dx,
\end{equation}
over $(\epsilon,\lambda)\in(-\delta,\delta)\times (-\delta,\delta)$, where $\delta>0$ be a small constant.  We put $v_\epsilon(\lambda;x)\equiv v^*=\frac{1-a_2}{c_2}$ if $\epsilon<0$ or $\lambda<0$.
\begin{proof}[Proof\nopunct]\emph{of Proposition} \ref{proposition3}.
It is easy to see that
\[G(0,0)=\lim_{\epsilon,\lambda\rightarrow 0}G(\epsilon,\lambda)=\int_0^L \Big(-1+\frac{1}{a_1+b_1\lambda e^{-rv_\epsilon}+c_1v_\epsilon}\Big)\lambda e^{-rv_\epsilon}=0.\]
For $\epsilon, \lambda>0$, we have that
\begin{equation}\label{438}
\begin{split}
\frac{\partial G(\epsilon,\lambda)}{\partial \lambda}=&\int_0^L\Big(r-\frac{(a_1+c_1v_\epsilon(\lambda;x))r+c_1}{(a_1+b_1\lambda e^{-rv_\epsilon(\lambda;x)}+c_1v_\epsilon(\lambda;x))^2 }\Big)\lambda e^{-rv_\epsilon(\lambda;x)}\frac{\partial v_\epsilon(\lambda;x)}{\partial \lambda} dx \\
&+\int_0^L\Big(\frac{a_1+c_1v_\epsilon(\lambda;x)}{(a_1+b_1\lambda e^{-rv_\epsilon(\lambda;x)}+c_1v_\epsilon(\lambda;x))^2 }-1\Big)e^{-rv_\epsilon(\lambda;x)}dx,
\end{split}
\end{equation}
where
\[
\frac{\partial v_\epsilon(\lambda;x)}{\partial \lambda}=-\frac{\partial \tilde w_\epsilon(\lambda;x)}{\partial \lambda}-\frac{\partial s_1}{\partial \lambda}=-\frac{\partial \tilde w_\epsilon(\lambda;x)}{\partial \lambda}-\frac{b_2e^{-r(v^*-s_1)}}{c_2}
\]
since $v_\epsilon=v^*-\tilde w-s_1$.  We claim that $\sup_{x\in[0,L]}\vert\frac{\partial w_\epsilon(\lambda;x)}{\partial \lambda}\vert$ is uniformly bounded in $\epsilon$ and $\lambda$.  According to (\ref{415}), we have
\[\epsilon^2 \frac{d^2 \frac{\partial \tilde w_\epsilon}{\partial \lambda}}{dx^2}+\tilde f_w(\lambda,\tilde w_\epsilon(\lambda;x))\frac{\partial \tilde  w_\epsilon}{\partial \lambda} +\tilde f_\lambda (\lambda; \tilde w_{\epsilon}(\lambda;x))=0,\]
where
\[\tilde f_\lambda=-\frac{b_2e^{r(\tilde{w}_\epsilon(\lambda;x)-v^*)}(v^*-\tilde{w}_\epsilon(\lambda;x))}{(1+b_2\lambda e^{r(\tilde{w}_\epsilon(\lambda;x)-v^*)}-c_2\tilde{w}_\epsilon(\lambda;x))^2},\]
and it is uniformly bounded in $\epsilon$ and $\lambda$.  Write $\tilde {\mathcal L}_\epsilon=\epsilon^2 \frac{d^2}{dx^2}+\tilde f_w(\lambda; \tilde w_\epsilon(\lambda;x))$.  By the standard perturbation arguments, we can show that $\tilde {\mathcal L}_\epsilon$ has an inverse $\tilde {\mathcal L}_\epsilon^{-1}$: $C([0,L])\rightarrow C([0,L])$ which is uniformly bounded in $\epsilon$.  Therefore $\frac{\partial \tilde w_\epsilon}{\partial \lambda}$ is bounded in $C([0,L])$.  Now we conclude from the Lebesgue Dominated Convergence Theorem $\frac{\partial G(\epsilon,\lambda)}{\partial \lambda}$ is continuous around $(\epsilon,\lambda)=(0,0)$; moreover, since $v^*=\frac{1-a_2}{c_2}\neq \frac{1-a_1}{c_1}$, we have
\begin{equation}\label{439}
\frac{\partial G(0,0)}{\partial \lambda}=\int_0^L\Big(\frac{a_1+c_1v^*}{(a_1+c_1v^*)^2 }-1\Big)e^{-rv^*}dx\neq 0,
\end{equation}
then Proposition \ref{proposition4} follows from the Implicit Function Theorem.
\end{proof}

\subsection{Boundary spike and boundary layer to the full system}
Now we prove Theorem \ref{theorem41} by showing that the solutions to (\ref{41}) perturb from its shadow system (\ref{46}) when $D_1$ is sufficiently large.  First of all, we let $s=\frac{1}{D_1}$, $\psi=v$, $\phi=ue^{rv}-\tau$, and $u=(\phi+\tau)e^{-rv}$, where $\tau$ is a constant.  Since
$\frac{\partial (\phi,\psi)}{\partial (u,v)}=
\begin{pmatrix}
e^{rv}& rue^{rv}\\
0& 1
\end{pmatrix}$
is invertible, (\ref{41}) is equivalent as $F(s,\phi,\tau,\psi)=0$, where
\begin{equation}\label{440}
F(s,\phi,\tau,\psi)=
\begin{pmatrix} \phi_{xx}-r\phi_x\psi_x+sP\{(-1+\frac{1}{a_1+b_1(\phi+\tau)e^{-r\psi}+c_1\psi})(\phi+\tau)\}\\
\epsilon^2\psi_{xx}+(-1+\frac{1}{a_2+b_2(\phi+\tau)e^{-r\psi}+c_2\psi})\psi\\
\int_0^L (-1+\frac{1}{a_1+b_1(\phi+\tau)e^{-r\psi}+c_1\psi})(\phi+\tau)e^{-r\psi}dx\end{pmatrix}
\end{equation}
We decompose $C([0,L])$ into $Y_0 \oplus Y_1$, where $Y_0=\{y\in C([0,L]) \vert \int_0^L udx=0\}$.  Define the projection operator $P: C^2([0,L]) \rightarrow Y_0$ by
\[P(u)=u-\frac{1}{L} \int_0^L u(x)dx.\]
\begin{proof}[Proof\nopunct]\emph{of Theorem} \ref{theorem41}.
Let $(\lambda_\epsilon,v_\epsilon(\lambda;x))$ be a solution of (\ref{46}), then we see that $F(0,0,\lambda_\epsilon, v_\epsilon(x))=0$.  Moreover, $F(s,\phi,\tau,\psi)$ is an analytic from $\mathbb R^+ \times C^2([0,L]) \times \mathbb R \times C^2([0,L])$ to $Y_0 \times C([0,L]) \times \mathbb R$ and its Frech\'et derivative with respect to $(\phi, \tau, \psi)$ at $(s,\phi,\tau,\psi)=(0,0,\lambda_\epsilon,v_\epsilon)$ is given by
\begin{equation}\label{441}
D_{(\phi, \tau, \psi)}F\vert_{(0,0,\lambda_\epsilon,v_\epsilon)}=\begin{pmatrix}
\frac{d^2}{dx^2}-rv_\epsilon^{'}\frac{d}{dx}&0&0\\
\tilde f_\lambda(\lambda_\epsilon;\tilde w_\epsilon) & \tilde f_\lambda(\lambda_\epsilon;\tilde w_\epsilon) &\epsilon^2\frac{d^2}{dx^2}+\tilde f_w(\lambda_\epsilon;\tilde w_\epsilon)\\
E_1& E_1&E_2
\end{pmatrix}
\end{equation}
where
\[E_1=\int_0^L \Big(\frac{a_1+c_1v_\epsilon}{(a_1+b_1\lambda_\epsilon e^{-r v_\epsilon }+c_1v_\epsilon)^2}-1 \Big)e^{-rv_\epsilon} \]
and
\[E_2=\int_0^L \Big(\frac{b_1r\lambda_\epsilon e^{-r v_\epsilon }-c_1}{(a_1+b_1\lambda_\epsilon e^{-r v_\epsilon }+c_1v_\epsilon)^2}+r\Big)\lambda_\epsilon e^{-rv_\epsilon}\]
The operator $\frac{d^2}{dx^2}-rv'_\epsilon(\lambda;x)\frac{d}{dx}$ is an isomorphism from $C^2([0,L])\cap Y_0$ onto $Y_0$.  Indeed, choosing $u$ in the kernel of this operator and writing $u=\sum_{k=0}^\infty S_k \cos \frac{k\pi x}{L}$, we have that $-S_k\frac{k \pi}{L}\cos \frac{k\pi x}{L}+r S_k v'_\epsilon(\lambda;x)\sin \frac{k\pi x}{L}=0$, which is impossible unless $S_k=0$.  Therefore $D_{(\phi, \tau, \psi)}F\vert_{(0,0,\lambda_\epsilon,v_\epsilon)}$ is bounded invertible if and only if the following problem
\begin{equation}\label{442}
\left\{
\begin{array}{ll}
\begin{split}
\epsilon^2\frac{d^2\eta_\epsilon}{dx^2}&+\tilde f_w(\lambda_\epsilon; \tilde w_\epsilon(\lambda,x))\eta_\epsilon+\tilde f_\lambda(\lambda_\epsilon; \tilde w_\epsilon(\lambda, x))\xi_\epsilon=0,\\
\xi_\epsilon\int_0^L \Big(&\frac{a_1+c_1v_\epsilon}{(a_1+b_1\lambda_\epsilon e^{-r v_\epsilon }+c_1v_\epsilon)^2}-1 \Big)e^{-rv_\epsilon} dx\\
&+\int_0^L \Big(\frac{b_1r\lambda_\epsilon e^{-r v_\epsilon }-c_1}{(a_1+b_1\lambda_\epsilon e^{-r v_\epsilon }+c_1v_\epsilon)^2}+r\Big)\lambda_\epsilon e^{-rv_\epsilon}\eta_\epsilon dx=0.
\end{split}
\end{array}
\right.
\end{equation}
has only the trivial solution $\xi_\epsilon= 0$ and $\eta_\epsilon(x)\equiv0$.  We argue by contradiction and without loss of generality, we assume that $\vert\xi_\epsilon\vert+\sup_{[0,L]}\vert\eta_\epsilon(x)\vert=1$. Since $\lambda_\epsilon\rightarrow 0$ as $\epsilon \rightarrow 0$, we infer from the second equation in (\ref{442}) that $\xi_\epsilon\rightarrow0$ as $\epsilon\rightarrow 0$.  On the other hand, since $\tilde f_\lambda(\lambda_\epsilon; \tilde w_\epsilon(\lambda,x))$ is bounded and $\tilde {\mathcal{L_\epsilon}}$  has a bounded inverse uniform in $\epsilon$, we conclude from the first equation that $\sup_{[0,L]} \vert\eta_\epsilon\vert \rightarrow 0$ as $\epsilon \rightarrow0$.  However, this contradicts our assumption.  Therefore $D_{(\phi, \tau, \psi)}F\vert_{(0,0,\lambda_\epsilon,v_\epsilon)}$ is nonsingular at $(0,0,\lambda_\epsilon, v_\epsilon(\lambda;x))$ and Theorem \ref{theorem41} follows from the Implicit Function Theorem.
\end{proof}
It is an interesting and also important mathematical question to study the stability of the single boundary spike solutions to the shadow system and the full system.  To this end, one needs more detailed asymptotic expansions on the perturbed solutions than we obtained and this is postponed to future study.  We refer to \cite{IWW,KW,KMW,KMW2,WW,WW2} and the references therein for stability analysis of spiky solutions of some closely related reaction--diffusion systems and/or the shadows systems.

\section{Conclusions and Discussions}\label{section5}
In this paper, we investigate the interspecific competition through a reaction--advection--diffusion system with Beddington--DeAngelis response functionals.  Species segregation phenomenon is modeled by the global existence of time-dependent solutions and the formation of stable boundary spike/layers over one--dimensional finite domains.
\begin{figure}[h!]
\centering
  \includegraphics[width=\textwidth]{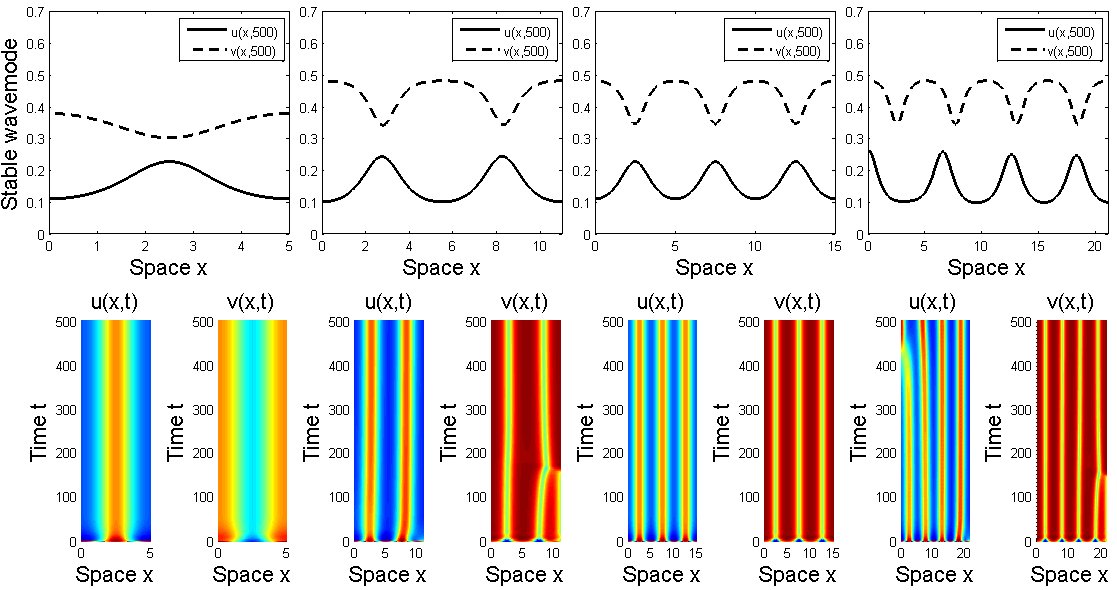}
 \caption{The sensitivity function $\phi(v)\equiv 1$ is selected.   Stable wave mode in the form of $\cos \frac{k_0\pi x}{L}$, where $k_0$ is given in Table \ref{table2}.  $\chi$ is chosen to be slightly larger than $\chi_0$ and the rest system parameters are chosen to be the same as in Table \ref{table2}.  Initial data are small perturbations of $(\bar u,\bar v)$.}\label{figure2}
\end{figure}
There are several findings in our theoretical results.  First of all, we prove that system (\ref{1})--(\ref{2}) admits global--in--time classical positive solutions which are uniformly bounded globally for any space dimensions.  For $\Omega=(0,L)$, we show that large directed dispersal rate $\chi$ of the species destabilizes the homogeneous solution $(\bar u,\bar v)$ and nonconstant positive steady states bifurcate from the homogeneous solution.  Moreover, we study the linearized stability of the bifurcating solutions which shows that the only stable bifurcating solutions must have a wavemode number $k_0$, which is a positive integer that minimizes bifurcating value $\chi_k$ in (\ref{314}).--see Theorem \ref{theorem32} and Theorem \ref{theorem33}.  That being said, the constant solution loses its stability to the wavemode $\cos\frac{k_0 \pi x}{L}$ over $(0,L)$.  Our result indicates that if the domain size $L$ is sufficiently small, then only the first wavemode $\cos\frac{\pi x}{L}$ can be stable, which is spatially monotone; moreover, large domain supports stable patterns with more aggregates than small domains.  We prove that the steady states of (\ref{1})--(\ref{2}) over $(0,L)$ admits boundary spike in $u$ and boundary layer in $v$ if $D_1$ and $\chi$ are comparably large and $D_2$ is small.  Therefore, we can construct multi-spike or multi-layer solutions to the system by reflecting and periodically extending the single spike over $\pm L$, $\pm 2L$,...  These nontrivial structures can be used to the segregation phenomenon through inter--specific competition.  Compared to the transition layer in \cite{WGY}, our results suggests the boundary spike solutions to (\ref{1}) is due to the Beddington--DeAngelis dynamics since a similar approach was adapted in \cite{WGY}.

In Table \ref{table2}, we list the wavemode numbers $k_0$ and the corresponding minimum bifurcation values for different lengths.  Figure \ref{figure2} gives numerical simulations on the evolutions of stable patterns over different intervals.  In Figure \ref{figure3}, we plot the spatial--temporal solutions to illustrate the formation of single boundary spike to $u$ and boundary layer to $v$.  The boundary spike and layer here correspond to those obtained in Theorem \ref{theorem31}.
\begin{table}[h!]
\centering
\begin{tabular}{|c|c|c|c|c|c|}
  \hline
    Domain size $L$& 3 & 5 & 7& 9& 11\\ \hline
    $k_0$&1&2&2&3&3\\ \hline
    $\chi_k$&9.9418&10.392&9.9120&9.9418&9.9647\\ \hline
    Domain size $L$& 13 & 15 & 17 & 19& 21\\ \hline
    $k_0$&4&5&5&6&6\\ \hline
    $\chi_k$&9.8872&9.9418&9.8937&9.8956&9.9120\\
  \hline
\end{tabular}
\caption{Stable wavemode number $k_0$ and the corresponding bifurcation value $\chi_{k_0}$ for different interval length.  The system parameters are $D_1=1, D_2=0.1, a_1=a_2=0.5, b_1=2, b_2=1$ and $c_1=0.5, c_2=1$.  According to Theorem \ref{theorem32}, the stable wavemode is $\cos \frac{k_0\pi x}{L}$.  Therefore, stable and nontrivial patterns must develop in the form of $\cos \frac{k_0\pi x}{L}$ if $\chi$ is chosen to be slightly larger than $\chi_{k_0}$.  We can also see that larger domains support bigger wavemode number.  Figure \ref{figure2} is given to illustrate the wavemode selection mechanism.}
\label{table2}
\end{table}
\begin{figure}[h!]
\centering
  \includegraphics[width=\textwidth]{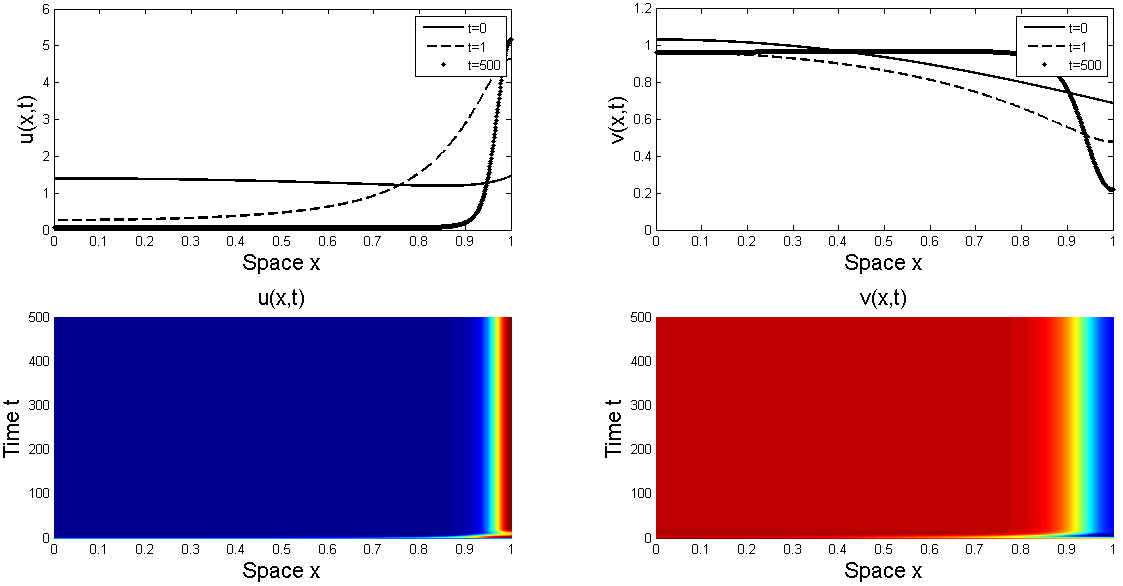}
 \caption{Formation of stable single boundary spike of $u$ and boundary layer of $v$.Diffusion and advection rates are chosen to be $D_1=5$, $\chi=30$, $D_2=5\times 10^{-3}$.  The rest system parameters are $a_1=0.2$, $b_1=0.8$, $c_1=0.1$ and $a_2=0.6$, $b_2=0.2$, $c_2=0.4$.  Initial data are $u_0=\bar u+0.5\cos\frac{2\pi x}{5}$ and $v_0=\bar v+0.5\cos\frac{2\pi x}{5}$, where $(\bar u,\bar v)=(0.933...,0.533...)$.}\label{figure3}
\end{figure}
We propose some problems for future studies.  System (\ref{1}) describes a situation that $u$ directs its disperse strategy over the habitat to deal with the population pressures from $v$, which moves randomly over the domain.  From the viewpoint of mathematical modeling, it is interesting to study the situation when both species takes active dispersals over the habitat.  Then system (\ref{1}) becomes a double--advection system with nonlinear kinetics.  Analysis of the new system becomes much more complicated now since no maximum principle is available for the second equation.  We also want to point out that it is also interesting to study the effect of environment heterogeneity on the spatial--temporal behaviors of the solutions.

Our bifurcation analysis is carried out around bifurcation point $(\bar u,\bar v,\chi_k)$.  It is interesting to investigate the global behavior of the continuum  of $\Gamma_k(s)$, denoted by $\mathcal C$.  According to the global bifurcation theory in \cite{R} and its developed version in \cite{SW}, there are three possible happenings for $\mathcal C$: (i) it is either not compact in $\mathcal X\times\mathcal X\times \mathbb R^+$, or it contains a point $(\bar u,\bar v,\chi_j)$, $j\neq k$.  Actually, by the standard elliptic embeddings, we can show that $\mathcal C$ is bounded in the axis of $\mathcal X\times \mathcal X$ if $\chi$ is finite.  Therefore, it either extends to infinity in the $\chi$-axis or it stops at $(\bar u,\bar v,\chi_j)$.  It is unclear to us which is the case for $\chi_k$.  Considering Keller--Segel chemotaxis models over $(0,L)$ without cellular growth,  some topology arguments are applied in \cite{CKWW,WX} etc. to show that all solutions on the first branch must either be monotone increasing or deceasing.  Therefore $\chi_1$ can not intersect at $(\bar u,\bar v,\chi_j)$ for $j\neq1$ since all the solutions around it must be non-monotone. This shows that $\Gamma_1$ extends to infinity in the $\chi$ direction.  However, the appearance of complex structures of the kinetics in our model inhabits this methodology.

Our numerical simulations suggest that these boundary spikes and layers are globally stable for a wide range of system parameters.  Rigorous stability analysis of these spikes is needed to verify our theoretical results.  It is also interesting to investigate the large--time behavior of (\ref{1})--(\ref{2}), which requires totally different approached from we apply in our paper.

\medskip
\medskip

\end{document}